\documentclass[11pt,a4paper,leqno,twoside]{amsart}%
\usepackage{amsmath,amssymb,amsthm,enumerate,hyperref,color}
\usepackage{graphicx}
\numberwithin{equation}{section}
\newtheorem{theorem}{Theorem}[section]

\newtheorem{proposition}[theorem]{Proposition}
\newtheorem{lemma}[theorem]{Lemma}

\newtheorem{remark}[theorem]{Remark}

\def\e{{\varepsilon}}

\def\l{{\lambda}}
\def\a{{\alpha}}

\def\de{\partial}

\newcommand{\R}{\mathbb{R}}
\newcommand{\N}{\mathbb{N}}
\newcommand{\Z}{\mathbb{Z}}

\newif\ifcomment \commentfalse
\def\commentON{\commenttrue}

\long\outer\def\BC#1\EC{\ifcomment \sloppy \par \# \ldots\dotfill
{\em #1} \dotfill \# \par \fi } \commentON

\newcommand{\remove}[1]{}

\def\sideremark#1{\ifvmode\leavevmode\fi\vadjust{\vbox to0pt{\vss
 \hbox to 0pt{\hskip\hsize\hskip1em
 \vbox{\hsize2.1cm\tiny\raggedright\pretolerance10000
  \noindent #1\hfill}\hss}\vbox to15pt{\vfil}\vss}}}%

\definecolor{cadmiumgreen}{rgb}{0.0, 0.42, 0.24}
\definecolor{darkgreen}{rgb}{0.0, 0.5, 0.2}
\definecolor{purple}{rgb}{0.5, 0.0, 0.5}

\newcommand{\taglia}{\color{cyan}}

\title[monotonicity under symmetry]{A monotonicity result under symmetry and Morse index constraints in the plane}
\author[F.~Gladiali]{Francesca Gladiali}
\thanks{This work was supported by Gruppo Nazionale per l'Analisi Matematica, la Probabilit\`a e le loro Applicazioni (GNAMPA) of the Istituto Nazionale di Alta Matematica (INdAM), by Prin-2015KB9WPT and by Fabbr.}
\address{Dipartimento di Chimica e Farmacia, Universit\`a di Sassari, via Piandanna 4, 07100 Sassari, Italy. \texttt{fgladiali@uniss.it}}

\begin{document}
\maketitle 
\begin{abstract}
This paper deals with solutions of semilinear elliptic equations of the type
\[
\left\{\begin{array}{ll}
-\Delta u = f(|x|, u) \qquad & \text{ in } \Omega, \\
u= 0 & \text{ on } \partial \Omega,
\end{array} \right.
\]
where $\Omega$ is a radially symmetric domain of the plane that can be bounded or unbounded. We consider solutions $u$ that are invariant by rotations of a certain angle $\theta$ and which have a bound on their Morse index in spaces of functions invariant by these rotations. We can prove that or $u$ is radial, or, else, there exists a direction $e\in \mathcal S$ such that $u$ is symmetric with respect to $e$ and it is strictly monotone in the angular variable in a sector of angle $\frac{\theta}2$. \\
The result applies to least-energy and nodal least-energy solutions in spaces of functions invariant by rotations and produces multiplicity results. 
\end{abstract}
{\small
\noindent{{{\bf{Keywords:}} semilinear elliptic equations, symmetry and monotonicity results, bounded and unbounded
domains, linearized equation, first eigenvalue and first eigenfunction, Morse index.}}}

\section{Introduction}
In this paper we study symmetry and monotonicity properties of classical solutions of semilinear elliptic problems of the type 
\begin{equation}\label{P}
-\Delta u = f(|x|, u) \qquad \  \text{ in } \Omega, 
\end{equation}
where $\Omega$ is a radially symmetric domain of $\R^2$
which can be bounded, in which case $\Omega$ is either a ball $B$ or an annulus centered at the origin, or 
can be unbounded in which case either $\Omega=\R^2$ or $\Omega=\R^2\setminus B$. When $\Omega$ is a ball $B$, an annulus or an exterior domain $\R^2\setminus B$ we also assume that $u$ satisfies Dirichlet boundary conditions
\begin{equation}\label{eq:bc}
u=0 \qquad \  \text{ on } \partial\Omega. 
\end{equation}
Throughout the paper 
$f:\bar \Omega\times \R\to \R$ is (locally) a $C^{0,\a}$-function whose first derivative with respect the second variable, that we denote hereafter $f'(|x|,s):=\frac{\partial f}{\partial s}(|x|,s)$, belongs to $C^{0,\a}$. \\
When $\Omega$ is a ball and $f$ is nonicreasing in the radial variable,
positive solutions to \eqref{P} and \eqref{eq:bc} are radially symmetric by the well known results of \cite{GNN}
where the moving plane method has been employed. Similar results hold also when $\Omega=\R^2$ at least under some decay assumption at infinity, see \cite{GNN2} or under some summability conditions as in \cite{CL}. 
However, when $\Omega$ is not convex
or when $f$ depends increasingly on the radial variable or when $u$ is a sign changing solution, the symmetry of all solutions does not hold anymore and indeed, in each of these cases, there are examples where radial and nonradial solutions coexist. 
We quote here the seminal paper \cite{SSW} dealing with positive solutions of the H\'enon problem, where $f(|x|,s)=|x|^\a s^p$ for $\a>0$ and $p>1$ in a ball, and it is proved that least energy solutions, namely solutions which minimize the energy functional, are nonradial provided $\a$ is large enough. 
 This symmetry breaking result is stated in dimension $N\geq 3$, but it holds also in the plane
as one can see in the examples in Section \ref{se:6}. Nonradial positive solutions have been found also in an annulus
and one can see the papers \cite{BN}, \cite{C}, \cite{Li} and \cite{GGPS}.
Also least energy nodal solutions are nonradial, when $\Omega$ is bounded as proved in \cite{AP}.  
\\
Nevertheless in some situations, or for a certain class of solutions, it is natural to expect that solutions inherit some of the symmetry of the domain, even if $\Omega$ is not convex, $u$ changes sign and $f$ is increasing in the radial variable. This is indeed the case for solutions of low Morse index, under some convexity assumption on $f$ that we shall make clear very soon, and it has been proved in \cite{Pacella} and \cite{PW} when $\Omega$ is bounded, and in \cite{GPW} when $\Omega$ is unbounded, that every solutions to \eqref{P} (and possibly \eqref{eq:bc}) of low Morse index is axially symmetric with respect to an axis passing through the
origin and nonincreasing in the polar angle from this axis, i.e. only depends on $r = |x|$ and $\theta=\arccos(x\cdot p)$, for a certain unit vector $p$, and $u$ is nonincreasing in $\theta$. This kind of symmetry 
is often called foliated Schwarz symmetry. See also the papers \cite{DP}, \cite{DGP1} and \cite{DGP2} where some extensions to the case of systems are considered.
\\ 
The foliated Schwarz symmetry for minimizers of certain variational problems has been obtained in \cite{SW} for positive solutions and in \cite{BWW} in the case on nodal solutions, using a completely different method based on 
symmetrization techniques.\\
Let us recall that the Morse index of a solution $u$ to \eqref{P} is the maximal dimension of a subspace 
of 
$C^1_0(\Omega)$ in which the quadratic form 
\[Q_{u,\Omega}(v,v):=\int_{\Omega}|\nabla v|^2-f'(|x|,u)v^2\ dx\]
 associated to the linearized operator
\[L_u(v):=-\Delta v -f'(|x|,u)v\]
is negative defined. Here and in the following by $f'(|x|,u)$ we mean $\frac{\partial f}{\partial s}(|x|,s) $.

\

In this paper we are interested in solutions which admit some invariances, namely they belong to suitable symmetric spaces, in the case when $\Omega$ is contained in the plane. Inspiring to the previous depicted papers \cite{PW} and \cite{GPW} we consider solutions which have low Morse index in these symmetric spaces and we prove, that under some convexity assumptions on $f$ or they are radial or they 
inherit only one extra-symmetry among the ones they can possess. This extra-symmetry of low Morse index solutions is what we think is the right generalization to the foliated Schwarz symmetry to this symmetric setting.
\\
Entering the details let us explain which type of symmetric spaces we are concern with. 
To this end, for any angle $\psi$ we denote by $R_\psi$ the rotation of angle $\psi$ in counterclockwise direction centered at the origin and by $\mathcal G_\psi$ the subgroup of $SO(2)$ generated by $R_\psi$. In particular we consider angles $\psi=\frac {2\pi}k$ with $k\in \N$, $k\geq 1$, so that $\mathcal G_{\frac {2\pi}k}$ is a proper subgroup of $SO(2)$. 
We say that 
a function $u$ defined in $\Omega$ is $k$-invariant if it satisfies
\begin{equation}\label{k-inv}
u(x)=u\left(g(x)\right) \ \ \text{ for all }x\in \Omega, \ \ \text{ for every } g\in \mathcal G_{\frac {2\pi}k}
\end{equation}
Next we denote by
$m_k(u)$ the $k$-invariant Morse index of $u$, namely the 
maximal dimension of a subspace of $C^1_{0,k}$, i.e. the subspace of $C^1_0(\Omega)$ given by functions that satisfy \eqref{k-inv},
in which the quadratic form $Q_{u,\Omega}$ is negative defined.\\ 

\begin{figure}[t!]
\includegraphics[scale=2]{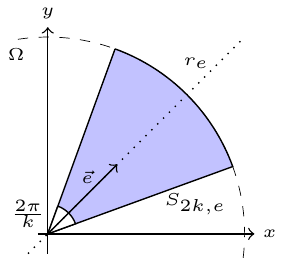}
\quad \quad \quad \quad \quad
\includegraphics[scale=2]{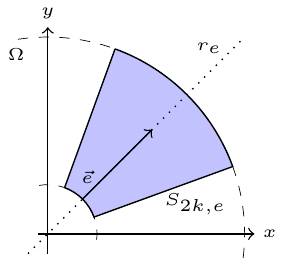}
\caption{The sector $S_{2k,e}$ when $\Omega$ is a ball or when $\Omega$ is an annulus}
\label{fig1}
\end{figure}

\

Since we are considering functions that satisfy \eqref{k-inv} it is enough to know them in any sector of $\Omega$ of angle $\frac{2\pi}k$. Then, for any direction $e\in\mathcal S$, $\mathcal S$ being the unit sphere, we denote by $S_{2k,e}$ the sector of $\Omega$ of angle $\frac {2\pi}k$, centered in the origin which has the straight line of direction $e$ passing through the origin, that we call $r_e$, as symmetry axis and lies in the halfplane $x\cdot e>0$, see Fig. \ref{fig1}.\\
In order to state our results we also need to introduce the two semisectors $S_{k,e}^+$ and $S_{k,e}^-$ that cover $S_{2k,e}$ and which are the part of $S_{2k,e}$ on one side of $r_e$ and on the other side respectively, see  Fig. \ref{fig2} and see Section \ref{se:2} where they are defined in a rigorous way.

\begin{figure}[t!]
\includegraphics[scale=2]{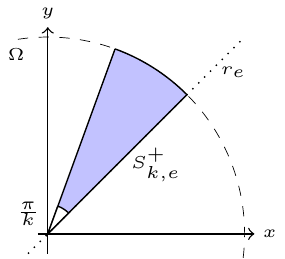}
\quad \quad \quad \quad \quad
\includegraphics[scale=2]{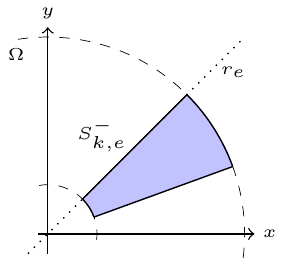}
\caption{The sector $S_{k,e}^+$ when $\Omega$ is a ball and the sector $S_{k,e}^-$ when $\Omega$ is an annulus}
\label{fig2}
\end{figure}

\

Now we can state the main results of this paper.
In particular we can prove the following result for solutions of low Morse index, in the case when the nonlinear term $f$ is convex in the second variable:
\begin{theorem}\label{teo:1}
Assume $\Omega$ is bounded. 
Let $u$
be a solution to \eqref{P} and \eqref{eq:bc} that satisfies \eqref{k-inv} and such that 
\[m_k(u)\leq 2.\]
Suppose $f(|x|,s)$ is  convex in the $s$-variable. 
Then, or $u$ is radial or else there exists a direction $e\in \mathcal S$ such that 
$u$ is symmetric with respect to $e$ in the sector $S_{2k,e}$ and
it is strictly monotone in the angular variable in the sectors $S_{k,e}^+$ and $S_{k,e}^-$.
\end{theorem}
 
\

Also in the case when the nonlinear term $f$ has a first derivative convex in the second variable we can prove an analogous result: 
\begin{theorem}\label{teo:2}
Assume $\Omega$ is bounded. 
Let $u$ be a solution to \eqref{P} and \eqref{eq:bc} that satisfies \eqref{k-inv} and such that 
\[m_k(u)\leq 2.\]
Suppose $f(|x|,s)$ has a convex derivative $f'(|x|,s)=\frac{\partial f}{\partial s}(|x|,s)$.
Then, or $u$ is radial or else there exists a direction $e\in \mathcal S$ such that 
$u$ is symmetric with respect to $e$ in the sector $S_{2k,e}$ and
it is strictly monotone in the angular variable in the sectors $S_{k,e}^+$ and $S_{k,e}^-$.
\end{theorem}

\

To our knowledge this is the first general result in this direction, i.e. in showing that solutions, constraint with some invariances, when they are nonradial, inherit a symmetry axis in their free sector $S_{2k,e}$, but they repulse additional symmetries, favoring the strict angular monotonicity in their free semisectors $S_ {k,e}^\pm$. It seems that this symmetry in $S_{2k,e}$ and monotonicity in $S_{k,e}^\pm$ is the exact generalization of the foliated Schwarz symmetry to the case of $k$-invariant functions or to the case of
sectors.\\
Let us explain the difficulties of the achievements. As in the previous papers the convexity assumptions on $f$ are needed in order to compare the quadratic form $Q_u$ with the quadratic form associated with the equation satisfied by the difference of two solutions to \eqref{P}. In particular we compare a solution $u$ with its reflection with respect to $r_e$ in the sector $S_{2k,e}$ and we move the direction $e$. 
\\
Next, the assumption on the $k$-Morse index of $u$ allows to say 
that at least in one of the semisectors $S_{k,e}^+$, for a suitable direction $e$, the first eigenvalue of the linearized operator $L_u$ is nonnengative. This implies that the angular derivative of $u$ in this sector $S_{k,e}^+$ or it is zero or it has a sign. However  it is possibile to prove that the first eigenvalue is nonnegative only when $u$ is symmetric with respect to $r_e$ in $S_{2k,e}$ so that the solution $u$ when nonradial admits this extra-symmetry. 
\\
Let us note in particular that the proof of Theorem \ref{teo:2} is technically complicated when $m_k(u)=2$ and also the sectors of amplitude $\frac{\pi}{2k}$ play an important role.

\

As one can see in the applications in Section \ref{se:6} the assumption $m_k(u)\leq 2$ is not very restrictive
and allows to treat the case of solutions of variational problems which minimizes the energy functional associated with \eqref{P} constraint either 
on the Nehari manifold or on the nodal Nehari manifold in the case of solutions that change sign. 
So also nodal solutions can be considered. According to the previous results in \cite{PW} and \cite{GPW} we believe moreover that this assumption should be optimal and  that allowing a higher $k$-Morse index would produce more symmetries in the sectors $S_{2k,e}$.\\
Note that the convexity assumption of Theorem \ref{teo:1} is satisfied by the exponential nonlinearities of Gelfand type, i.e., $f(r,s)=\l V(r)e^s$ and, for $s>0$ by the Lane-Emden type, i.e. $f(r,s)=|s|^{p-1}s$ for $p>1$ and by their extension of H\'enon type, namely $f(r,s)=r^\a e^s$ and $f(r,s)=r^\a |s|^{p-1}s$ for $\a>0$. \\
The convexity assumption in Theorem \ref{teo:2}, instead, allows to deal with nodal solutions to the Lane-Emden problem, for $p\geq 2$ and with positive or nodal solutions to the sinh-Poisson problem, namely when $f(r,s)=\e(e^s-e^{-s})$ as well as their extensions of H\'enon type.
Section \ref{se:6} provides a broad range of problems that satisfy the convexity assumptions and to which these topics can be applied and some multiplicity results are produced.

\

Finally in the case when the domain $\Omega$ is not bounded we can prove the same type of monotonicity results. 
\begin{theorem}\label{teo:3}
Assume $\Omega$ is unbounded. 
Let  $u$ be a solution to \eqref{P} and possibly \eqref{eq:bc} that satisfies \eqref{k-inv} and such that 
$|\nabla u|\in L^2(\Omega)$ and 
\[m_k(u)\leq 2.\]
Assume further 
that $f$ or $f'$ are convex in the second variable.  Then, or $u$ is radial or else there exists a direction $e\in \mathcal S$ such that 
$u$ is symmetric with respect to $e$ in $S_{2k,e}$ and
it is strictly monotone in the angular variable in the sectors $S_{k,e}^+$ and $S_{k,e}^-$.
\end{theorem}
Let us remark that very few results are available when $f$ depends on the radial varaibale, or when the solutions change sign or when the underlying domain is $\R^2\setminus B$ and we think that Theorem \ref{teo:3} is a first step in this direction. \\
When passing from the bounded to the unbounded case we have to take into account either the fact that the first eigenvalue and the first eigenfunction
is not defined anymore in an unbounded domain either that some of the functions constructed do not have the right sommability. Then we have to divide the problem considering first a bounded section of the sectors $S_{2k,e}$ and then looking at the unbounded part of the sectors. In this last issue it is important that, in a certain sense, the bound on the $k$-Morse index, means that the maximum principle should hold in the unbounded part.

\

Finally let us observe that we are confident that similar results should hold also in higher dimensions. Nevertheless while the $k$-invariance in \eqref{k-inv} is very natural for solutions in a radially symmetric domain of the plane, since they are the same invariances of the Spherical Harmonics, in higher dimension the Spherical Harmonics are far more complicated and many different invariances should be taken into account.


\

\

\section{Notations and preliminary results}\label{se:2}
In this section we introduce all the notations we need to prove the main Theorems and we give some preliminary results in the case when $\Omega$ is bounded.
Let us denote by $e\in \mathcal S$ any direction, $e=e_\psi=(\cos\psi,\sin\psi)$  with $\psi\in[0,2\pi)$ whose orthogonal vector is given by $e_{\Tiny\perp}:=(-\sin \psi, \cos\psi)$, and by $r_e$ or $r_{\psi}$ the straight line passing through the origin of direction $e$, namely $r_e:=x\cdot e_{\Tiny\perp}=0$.\\

Let $\sigma_e:\R^2\to \R^2$ be the reflection with respect the line $r_e$, i.e. $\sigma_e(x):=x-2(x\cdot e_{\Tiny\perp})e_{\Tiny\perp}$ for every $x\in \Omega$ and, if $u$ is any solution to \eqref{P} we let
\begin{equation}\label{eq:difference}
w_e(x):=u(\sigma_e(x))-u(x)
\end{equation}
the difference between the reflection of $u$ with respect to $r_e$ and $u$.\\
Since we are interested in solutions with some invariances, as explained in the Introduction, 
for any angle $\psi$ we denote by $R_\psi$ the rotation of angle $\psi$ in counterclockwise direction centered at the origin and by $\mathcal G_\psi$ the subgroup of $SO(2)$ generated by $R_\psi$. In particular we consider angles $\psi=\frac {2\pi}k$ with $k\in \N$, $k\geq 1$, so that $\mathcal G_{\frac {2\pi}k}$ is a proper subgroup of $SO(2)$. Note that for $k=1$, $\mathcal G_{2\pi}=\{I\}$ is the trivial subgroup. \\
We say that 
a function $u$ defined in $\Omega$ is $k$-invariant if it satisfies
\begin{equation}\tag{\ref{k-inv}}
v(x)=v\left(g(x)\right) \ \ \text{ for all }x\in \Omega, \ \ \text{ for every } g\in \mathcal G_{\frac {2\pi}k}
\end{equation}
Since we will consider functions in that satisfy \eqref{k-inv} it is enough to know them in any sector of $\Omega$ of angle $\frac{2\pi}k$. Then, for any direction $e\in\mathcal S$ we denote by $S_{2k,e}$ the circular sector of $\Omega$ of angle $\frac {2\pi}k$, centered in the origin which has $r_e$ has symmetry line and lies in the halfplane $x\cdot e>0$, see Fig. \ref{fig1} in the Introduction.\\

In order to understand the behavior, and the symmetries of the solutions we need to work also in sectors of amplitude $\frac{\pi}k$ and so, for any $e\in\mathcal S$, such that $e=(\cos\psi,\sin\psi)$ we let
\[S_{k,e}^+:=
\{(r,\theta)\in S_{2k,e} : \psi<\theta<\psi+\frac \pi k\}\]
and by
\[S_{k,e}^-:
\{(r,\theta)\in S_{2k,e} : \psi-\frac \pi k<\theta<\psi\}\]
see Fig. \ref{fig2} in the Introduction.
Observe that $\partial S_{k,e}^+=\Gamma_{1,e}\cup \Gamma_{2,e}\cup\Gamma_{3,e}$ where $\Gamma_{1,e}$ is contained in $\partial \Omega$, $\Gamma_{2,e}$ is contained in $r_e$ while $\Gamma_{3,e}$ is contained in the line $r_{\psi+\frac \pi k}$ if $e=(\cos\psi, \sin\psi)$, 
see Fig. \ref{fig3}.

\begin{figure}[t!]
\includegraphics[scale=2]{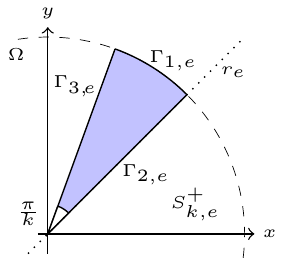}
\quad \quad \quad \quad \quad
\includegraphics[scale=2]{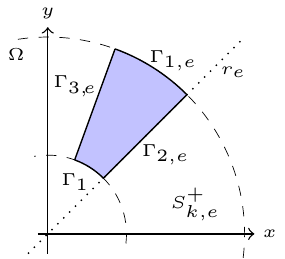}
\caption{The sector $S_{k,e}^+$ and its boundary when $\Omega$ is a ball or when $\Omega$ is an annulus}
\label{fig3}
\end{figure}

When $e=e_\psi=(\cos\psi,\sin\psi)$ we will use also the notation $S_{k,\psi}^+$, $S_{k,\psi}^-$, $S_{2k,\psi}$ and $w_\psi$ to denote respectively $S_{k,e_\psi}^+$, $S_{k,e_\psi}^-$, $S_{2k,e_\psi}$ and $w_{e_\psi}$.\\

By elliptic regularity theory $u\in C^{3,\beta}(\bar \Omega)$ for some $\beta>0$. In particular $f'(|x|,u)=\frac{\partial f}{\partial u}(|x|,u)\in C^{0,\alpha}(\bar \Omega)$ for some $\alpha>0$.
We can then define the linearized operator at a solution $u$
\[L_u:=-\Delta-f'(|x|,u)I\]
and, for any $D\subseteq \Omega$ the quadratic form associated with $L_u$ in $D$, namely
\[Q_{u,D}(v,w):=\int_D \nabla v\nabla w-f'(|x|,u)vw\]
which is defined for $v,w\in H^1_0(D)$. Next, for any direction $e\in \mathcal S$ and for any $x\in \Omega$  we denote by
\[V_e(x):=\int_0^1 f' ( |x|, tu(\sigma_e(x))+(1-t)u(x)) \ dt,\]
\[V_{es}(x)=\frac 12 \left(f'(|x|,u(x))+f'(|x|,u(\sigma_e(x)))\right)\]
and by
\[L_e:=-\Delta -V_e(x)I \ \ \ \ \ \  \ \  \ L_{es}:=-\Delta -V_{es}(x)I\]
the corresponding linear operators, which are associated with the quadratic forms
\[Q_{e,D}(v,w):=\int_D \nabla v\nabla w -V_e(x)vw \ \ \ \ \ \  \ Q_{es,D}(v,w):=\int_D \nabla v\nabla w  -V_{es}(x)vw
.\]
For any $e\in \mathcal S$, the function $w_e$ satisfies 
\[L_e w_e=0.\] 
Moreover $w_e=0$ on $\partial \Omega$ and on $r_e$ by construction.
Further, if $u$ satisfies \eqref{k-inv} then $w_e=0$ also on $r_{\psi+\frac \pi k}$ by the rotation invariance of $u$, showing that for any $e\in \mathcal S$, $w_e=0$ on $\partial S_{k,e}^+$ and on $\partial S_{k,e}^-$.\\
Observe also that, if $w_e\equiv 0$ in $D$ then $L_e=L_{es}=L_u$ in $D$.\\
Next, for any $D\subseteq \Omega$ and for any linear operator $L$ we denote by $\l_1(L,D)$ the first eigenvalue of the linear operator $L$ in $D$ with Dirichlet boundary conditions.

\

First we can prove the following
\begin{proposition}\label{prop-2.1}
Assume $\Omega$ is bounded.
Let $u$ be a solution to \eqref{P} and \eqref{eq:bc} that satisfies \eqref{k-inv}.
Suppose further that 
there exists a direction $e\in \mathcal S$ such that
\begin{equation}\label{eq:w-e-segno-costante}
w_e>0 \ \ \hbox{ in }S_{k,e}^+ \ \ \hbox{ or }\ w_e<0 \ \ \hbox{ in }S_{k,e}^+.
\end{equation}
Then there exists another direction $e'\in \mathcal S$ such that $w_{e'}\equiv 0$ in 
$S_{k,e'}^+$
(i.e. $u(\sigma_{e'}(x))=u(x)$ in $S_{k,e'}^+$) and
\begin{equation}\label{eq:primo-autov-nonnegativo}
\l_1(L_u, S_{k,e'}^+)\geq 0.
\end{equation}
\end{proposition}
\begin{proof}
We can assume w.l.o.g. that $w_e>0$ in $S_{k,e}^+$ and that, up to a rotation, $e=(1,0)=(\cos \psi_0,\sin \psi_0)$ with $\psi_0=0$. Now we consider other directions $e_\psi:=(\cos \psi,\sin \psi)$ for $\psi>0$. Note that, if $w_\psi>0$ in $S_{k,\psi}^+$ for some $\psi>0$ then 
$\l_1(L_{e_\psi}, S_{k,\psi}^+)=0$, since $w_{\psi}$ solves $L_{e_\psi}(w_{\psi})=0$ in $S_{k,\psi}^+$ and $w_\psi=0$ on $\partial S_{k,\psi}^+$ by previous considerations. 
Next we define 
\begin{equation}\label{tilde-psi}
\tilde \psi=\sup\{\psi\in [0,\frac\pi k): \hbox{ s.t. } w_{\psi'}\geq  0 \ \hbox{ in } S_{k,\psi'}^+ \text{ for all }\psi'\in[0,\psi)\}\end{equation}
By continuity $\l_1(L_{e_{\tilde \psi}}, S_{k,\tilde \psi}^+)=0$
 and $w_{\tilde \psi}\geq 0$ in $S_{k,\tilde \psi}^+$. This implies that $\tilde \psi<\frac \pi k$, because, by the rotation invariance of $u$, we have 
$w_{\frac \pi k}=-w_e<0$ in $S_{k,\frac \pi k}^+$. \\
We would like to prove that $w_{\tilde \psi} \equiv  0$ in $S_{k,\tilde \psi}^+$. \\
Assume, by contradiction, that $w_{\tilde \psi}\neq 0$. Then, necessarely $w_{\tilde \psi}>0$ in $S_{k,\tilde \psi}^+$. Indeed, by definition of $\tilde \psi$ we have $w_{\tilde \psi}\geq 0$, and since $w_{\tilde \psi}$ satisfies $L_{e_\psi} w_\psi=0$ then the strong maximum principle implies the assertion. \\
Now let $K\subset S_{k,\tilde \psi}^+$ be a compact set such that $S_{k,\tilde \psi}^+\setminus K$ has small measure to allow the weak maximum principle hold for the operator $L_{e_{\tilde \psi+\e}}$ in $S_{k,\tilde \psi+\e}^+\setminus K$ for sufficiently small $\e$.\\
Since $w_{\tilde \psi}>\eta>0$ in $K$ for some $\eta>0$, then
\[w_{\tilde \psi+\e}>\frac \eta 2>0  \ \ \hbox{ in }K\subset  S_{k,\tilde \psi+\e}^+\]
for sufficiently small $\e>0$, while
\[w_{\tilde \psi+\e}>0  \ \hbox{ in }S_{k,\tilde \psi+\e}^+\setminus K\]
by the weak and the strong maximum principle in $S_{k,\tilde \psi+\e}^+\setminus K$. \\
Therefore $w_{\tilde \psi+\e}>0$ in $S_{k,\tilde \psi+\e}^+$ contradicting the definition of $\tilde \psi$. Hence $w_{\tilde \psi}$ must be identically zero in $S_{k,\tilde \psi}^+$. This implies that $0=\l_1(L_{e_{\tilde \psi}}, S_{k,\tilde\psi}^+)=\l_1(L_u, S_{k,\tilde\psi}^+)$ concluding the proof.
\end{proof}

\

Next we have 
\begin{proposition}\label{prop-2.2}
Assume $\Omega$ is bounded.
Let $u$ be a solution to \eqref{P} and \eqref{eq:bc} that satisfies \eqref{k-inv}.
Suppose further that there exists a direction $e\in \mathcal S$ such that $w_e\equiv 0$ in $S_{k,e}^+$ and 
\begin{equation}\label{eq:primo-autov-pos}
\l_1(L_u, S_{k,e}^+)\geq 0.
\end{equation}
Then or $u$ is radial or it is strictly monotone in the angular variable in $S_{k,e}^+$.
\end{proposition}
\begin{proof}
We can assume w.l.o.g. that $e=(1,0)$ and we use the polar coordinates $(x_1,x_2)=(r\cos \theta, r\sin \theta)$. As said before, if $u$ solves \eqref{P} then $u\in C^{3,\beta}(\bar\Omega)$ and the derivative of $u$ with respect to the angular variable $\theta$, that we denote $u_{\theta}$, satisfies
\[
-\Delta u_\theta=f'(|x|,u)u_\theta\ \  \hbox{ in }S_{k,e}^+\]
Moreover $u_\theta=0$ on $\partial \Omega$, since $u$ satisfies zero Dirichlet boundary conditions on $\partial \Omega$,
and $u_\theta=0$ on the $\Gamma_{2,e}$ by the symmetry of $u$ with respect to $r_e$. 
Further, the symmetry of $u$ with respect to $r_e$ together with the invariance by rotations of angle $\frac{2\pi}k$ implies also that $u$ is symmetric with respect to the line $\theta={\frac \pi k}$ which is part of the boundary of $S_{k,e}^+$. Hence $u_\theta=0$ also on $\theta={\frac \pi k}$ meaning that $u_\theta=0$ on $\partial S_{k,e}^+$. \\
Now, since $\l_1(L_u,S_{k,e}^+)\geq 0$ we have or that $u_\theta=0$ in $S_{k,e}^+$, meaning that $u$ is radial, or, else that $u_\theta$, which is nonzero, is a first eigenfunction of $L_u$ in $S_{k,e}^+$ and so it has one sign in $S_{k,e}^+$, namely it is positive or negative in $S_{k,e}^+$. 
This shows that if $u$ is not radial it is strictly monotone in the polar angle $\theta$ and concludes the proof.
\end{proof}

By the symmetry assumption on $u$, whenever $u$ is nonradial then it is strictly monotone in the polar angle $\theta$ also in the sector $S_{k,e}^-$.

\

Next, for every real valued function $g$, we let $g^+(x)=\max\{g(x), 0\}$ and $g^-(x)=\min\{g(x),0\}$ denote the 
positive and negative part of $g$, respectively. For every $D\subseteq \Omega$,  $\chi_D$ stands for the characteristic function of $D$. If $D$ is a domain, we let $C^1_0(D)$ denote the space of all $C^1$-functions on $D$ with compact support strictly contained in $D$.\\ 
Finally, as explained in the Introduction, we denote by
$m_k(u)$ the $k$-Morse index of $u$, namely the 
maximal dimension of a subspace of $C^1_{0,k}$
in which the quadratic form $Q_{u,\Omega}$ is negative defined. Equivalently, when $\Omega$ is bounded, it is the number, counted with multiplicity, of the negative eigenvalues of $L_u$ with corresponding eigenfunction $k$-invariant.

\section{The case of $f$ convex}
In this section we prove the monotonicity and symmetry results in the case when $\Omega$ is bounded and $f$ is convex with respect to the second variable. Under this assumption
\[f(|x|,u(\sigma_e(x)))-f(|x|,u(x))\geq f'(|x|,u(x))w_e(x)\]
and, 
for any direction $e\in \mathcal S$, the function $w_e$ satisfies the inequality
\begin{equation}\label{eq:ineq-we}
-\Delta w_e-f'(|x|,u)w_e\geq 0 \ \ \hbox{ in }S_{k,e}^+\end{equation}
and also the boundary conditions
$w_e=0$ on $\partial S_{k,e}^+$.
First we prove the following
\begin{lemma}\label{lemma-fine-dim}
Assume $\Omega$ is bounded. 
Assume further that one among $\l_1(L_u,S_{k,e}^+)$ and $\l_1(L_u,S_{k,e}^-)$ is nonnegative. Then or $w_e\equiv 0$ in $S_{k,e}^+$ and $\l_1(L_u,S_{k,e}^+)\geq 0$
or $w_e$ has a sign in $S_{k,e}^+$.
\end{lemma}
\begin{proof}
We can assume w.l.o.g. that $\l_1(L_u, S_{k,e}^+)\geq 0$. \\
If $w_e\equiv 0 $ in $S_{k,e}^+$
then we are done. Assume else that $w_e\neq 0 $ in $S_{k,e}^+$. We want to prove that $w_e$ has a sign in $S_{k,e}^+$. Assume not, then $w_e^-\neq 0$, and $w_e^-=0$ on $\partial S_{k,e}^+$, since $w_e=0$ on $\partial S_{k,e}^+$ by previous considerations.
Multiplying \eqref{eq:ineq-we} by $w_e^-$ and integrating over $S_{k,e}^+$ we have
\[Q_{u, S_{k,e}^+}(w_e^-,w_e^- )\leq 0\]
which, together with the relation 
$\l_1(L_u, S_{k,e}^+)\geq 0$ implies
\[Q_{u,  S_{k,e}^+ }(w_e^-,w_e^-)=0\]
meaning that $w_e^-$ reaches the infimum of $Q_{u,  S_{k,e}^+ }(v,v)$ in $H^1_0(S_{k,e}^+)$ and hence it solves 
\[L_u(w_e^-)=0 \ \ \hbox{ in }S_{k,e}^+.\]
Now, since $w_e^-\leq 0$ the strong maximum principle implies either that $w_e^-\equiv 0$ in $S_{k,e}^+$, which is not possible by assumption, or that $w_e^-<0$ in $S_{k,e}^+$ which implies $w_e<0$. This concludes the proof.
\end{proof}

\

We are now in position to prove the main point:

\begin{proposition}\label{prop-3.2}
Assume $\Omega$ is bounded. Let $u$ be a solution to \eqref{P} and \eqref{eq:bc} that satisfies \eqref{k-inv}
and such that 
$m_k(u)\leq 2$. Then there exists a direction $e\in \mathcal S$ such that, or $w_e\equiv 0$ in $S_{k,e}^+$ and $\l_1(L_u,S_{k,e}^+)\geq 0$ or $w_e$ has a sign in $S_{k,e}^+$.
\end{proposition}
This proposition, together with Propositions \ref{prop-2.1} and \ref{prop-2.2} then implies Theorem \ref{teo:1} and concludes this case.
\begin{proof}
Let us denote by
$\varphi_1$ and  $\varphi_2$ two $k$-invariant eigenfunctions of the linearized operator $L_u$ 
orthogonal in $L^2(\Omega)$
corresponding to the two negative eigenvalues. \\
For any direction $e\in \mathcal S$ let us denote by $\varphi_e^+$ the unique first positive eigenfunction of $L_u$ in $S_{k,e}^+$ with Dirichlet boundary conditions, which corresponds to $\l_1(L_u, S_{k,e}^+ )$ and by $\varphi_e^-$ the unique first positive eigenfunction of $L_u$ in $S_{k,e}^-$ with Dirichlet boundary conditions, which corresponds to $\l_1(L_u, S_{k,e}^- )$, normalized such that
$\int_{S_{k,e}^+}
\left(\varphi_e^+\right)^2\ dx=\int_{S_{k,e}^-}
\left(\varphi_e^-\right)^2\ dx=1$
and by $\tilde \varphi_e^\pm$ their extension in all of $\Omega$ by the invariance by rotations of angle $\frac{2\pi}k$. 
We want to prove that there exists a direction $e\in \mathcal S$ for which at least one, among $\l_1(L_u, S_{k,e}^+ )$ and $\l_1(L_u, S_{k,e}^- )$ is nonnegative, so that the result follows by Lemma \ref{lemma-fine-dim}. \\
Define the function
 \[\xi_e:=\left(\frac {\int _{  S_{k,e}^-} \varphi_{e}^-\varphi_1  }{\int _{  S_{k,e}^+} \varphi_{e}^+\varphi_1  }\right)^{\frac 12} \varphi_e^+-\left(\frac {\int _{  S_{k,e}^+} \varphi_{e}^+\varphi_1}{\int _{  S_{k,e}^-} \varphi_{e}^-\varphi_1  } \right)^{\frac 12}\varphi_{e}^-=A_e\varphi_e^+-B_e\varphi_{e}^-\]
 which is supported in $S_{2k,e}$, and denote by $\tilde \xi_e$ its extension in all of $\Omega$ by the invariance by rotations of angle $\frac{2\pi}k$.
 Then, by construction
\[\begin{split}
\int_\Omega \tilde\xi_e\varphi_1\ dx&=k\int_{S_{2k,e}}\xi_e\varphi_1 \ dx=k
\left[A_e\int _{S_{k,e}^+}\varphi_e^+\varphi_1 \ dx-B_e\int _{S_{k,e}^-}\varphi_{e}^-\varphi_1 \ dx\right]\\
&=k\left[\left(\int _{  S_{k,e}^-} \varphi_{e}^-\varphi_1\int _{S_{k,e}^+}\varphi_e^+\varphi_1dx\right)^{\frac 12}-\left(\int _{S_{k,e}^+}\varphi_e^+\varphi_1
\int _{  S_{k,e}^-} \varphi_{e}^-\varphi_1dx\right)^{\frac 12}\right]=0
\end{split}
\]
which means that $\tilde \xi_e$ is 
orthogonal in $L^2(\Omega)$ to $\varphi_1$ for every $e\in \mathcal S$. Assume $e=e_\psi=(\cos \psi, \sin \psi)$, with $\psi\in [0,2\pi)$.
Let us consider the map
\begin{equation}
h: [0,2\pi)\to \R \ \ , \ \ h(\psi):=\int _{S_{2k,\psi}}\xi_\psi\varphi_2 \ dx \end{equation}
which is continuous with respect to the angle $\psi$. Here $\xi_\psi$ stands for $\xi_{e_\psi}$. We want to prove that there exists $\psi\in [0,\frac \pi k)$ such that $h(\psi)=0$. 
If $h(0)=0$ then we are done.
If else $h(0)\neq 0$, observe that $S_{k, \frac \pi k}^-=S_{k, 0}^+$ and that $S_{k, \frac \pi k}^+$ can be obtained rotating the sector $S_{k, 0}^-$ by $\frac {2\pi }k$, namely $S_{k, \frac \pi k}^+=R_{\frac{2\pi}k}(S_{k, 0}^-)$. Then $\varphi_{\frac \pi k}^+=\tilde \varphi_{0}^-$ by the uniqueness of the first positive eigenfunction of norm $1$ and $\varphi_{\frac \pi k}^-=\varphi_{0}^+$.\\
Moreover
\[
A_{e_{\frac \pi k}}=\left(\frac {\int _{  S_{k,\frac \pi k }^-} \varphi_{\frac \pi k}^-\varphi_1  }{\int _{  S_{k,\frac \pi k}^+} \varphi_{\frac \pi k}^+\varphi_1  }\right)^{\frac 12} =\left(\frac {\int _{  S_{k,0}^+} \varphi_{0}^+\varphi_1  }{\int _{  S_{k,0}^-} \varphi_{0}^-\varphi_1  }\right)^{\frac 12} =B_{e_{0}}
\]
where we used the fact that $\varphi_1$ is $k$-invariant. In the same manner 
\[B_{e_{\frac \pi k}}=A_{e_{0}}\]
It is then easy to see that $\tilde \xi_{\frac \pi k}=-\tilde\xi_0$, that implies 
$h(\frac \pi k)=-h(0)$
since $\varphi_2$ is $k$-invariant.
Since $h$ is a continuous function then it has one zero between $0$ and $\frac \pi k$. This means that there exists a direction $e'\in \mathcal S$ such that
$\tilde \xi_{e'}$ is orthogonal to $\varphi_2$ in $L^2(\Omega)$, since $\int_\Omega \tilde \xi_{e'}\varphi_2 =k\int_{S_{2k,e'}} \xi_{e'}\varphi_2 =0$. 
Since $\varphi_1, \varphi_2$ and  $\tilde \xi_{e'}$ are $k$-invariant and  $m_k(u)=2$ then 
\[Q_{u,\Omega}(\tilde \xi_{e'},\tilde \xi_{e'})\geq 0.\]
Finally by the definition of $\tilde \xi_{e'}$ we have 
\[\begin{split}
&Q_{u,\Omega}(\tilde \xi_{e'},\tilde \xi_{e'})=\\
&=k\left[A_{e'}^2 \left(\int_{S_{k,e'}^+}|\nabla \varphi_{e'}^+|^2-f'(|x|,u)(\varphi_{e'}^+)^2\right)+B_{e'}^2 \left(\int_{S_{k,e'}^-}|\nabla \varphi_{e'}^-|^2-f'(|x|,u)(\varphi_{e'}^-)^2\right)\right]\\
&=kA_{e'}^2\l_1(L_u,S_{k,e'}^+)+kB_{e'}^2\l_1(L_u,S_{k,e'}^-)
\geq 0
\end{split}
\]
from which it follows that al least one among $\l_1(L_u,S_{k,e'}^+)$ and $\l_1(L_u,S_{k,e'}^-)$ is nonnegative.
\end{proof}

\begin{remark}\label{rem:1}
When a $k$-invariant solution $u$ is not radial, then, by Theorem \ref{teo:1} it is symmetric with respect to a direction $e=e_\psi$ in $S_{2k,e}$ and by the invariance \eqref{k-inv} it is symmetric also with respect to the directions $e_{\psi+\frac \pi k}$ and also $e_{\psi+\frac {h\pi}k}$ for $h=1,\dots,2k$. Moreover 
in each sector $S_{k, \psi+\frac {h\pi}k}^\pm$ it is strictly monotone with respect to the angular variable and consecutive sectors see the angular derivative change sign. \\
Then, the maxima and minima of $u$ either belong to the symmetry axes or they are placed in the origin.
Both this configurations can appear, in particular for sign changing solutions. Indeed a nodal $k$-invariant solution can have, as an example, either $k$ maxima and $k$ minima placed alternately along the directions $(\cos(\psi+\frac {h\pi}k), \sin(\psi+\frac {h\pi}k))$ at an angular distance of $\frac \pi k$, or else only $1$ maximum in the origin and $k$ minima placed along the directions $(\cos(\psi+\frac {2h\pi}k), \sin(\psi+\frac {2h\pi}k))$ at an angular distance of $\frac {2\pi} k$.
\end{remark}

\

\section{the case of $f'$ convex}
In this section we consider the case when $f'(|x|,s)=\frac{\partial f}{\partial s}(|x|,s)$ is convex with respect to the second variable and we prove Theorem \ref{teo:2}. Under this assumption we have for any $x\in \Omega$
\begin{equation}\label{eq:confronto-V}
\begin{split}
V_e(x)&\leq \int_0^1\left[tf'(|x|,u(x))+(1-t)f'(|x|,u(\sigma_e(x)))\right]\ dt\\
& =\frac 12 \left[f'(|x|,u(x))+f'(|x|,u(\sigma_e(x)))\right]=V_{es}(x).
\end{split}\end{equation} 
First, we can prove the following:
\begin{proposition}\label{prop-4.1}
Assume $\Omega$ is bounded. Let $u$ be a solution to \eqref{P} and \eqref{eq:bc} that
satisfies \eqref{k-inv} and such that
$m_k(u)\leq 2$. Suppose further that 
$f'$ is convex with respect to the second variable. Then there exists a direction $e\in \mathcal S$ such that, or $w_e\equiv 0$ in $S_{k,e}^+$ and $\l_1(L_u,S_{k,e}^+)\geq 0$ or $w_e$ has one sign in $S_{k,e}^+$.
\end{proposition}
This proposition, together with Propositions \ref{prop-2.1} and \ref{prop-2.2} then  implies Theorem \ref{teo:2} and concludes this case.

\begin{proof}
Let us denote by
$\varphi_1$ and  $\varphi_2$ two k-invariant
eigenfunctions of the linearized operator $L_u$, orthogonal in $L^2(\Omega)$, corresponding to the two negative eigenvalues.\\
Define 
\begin{equation}\label{eq:defS}
S_*:=\{e\in\mathcal S : w_e\equiv 0 \hbox{ in }S_{k,e}^+ \hbox{ and }  \inf _{\psi\in C^1_0(S_{k,e}^+)} Q_{u, S_{k,e}^+}(\psi,\psi)<0
\}
\end{equation}
\noindent {\bf Case 1:} $S_*=\emptyset$\\
Assume  by contradiction that for every $e\in \mathcal S$ the function $w_e$ changes sign in $S_{k,e}^+$. In this case for every $e\in \mathcal S$ we can define the functions
\begin{equation}\label{eq:w1e2}
w^1_e(x)=w_e^+\chi_{S_{k,e}^+}-w_e^- \chi_{S_{k,e}^-}\ \ \ \ 
w^2_e(x)=-w_e^-\chi_{S_{k,e}^+}+w_e^+ \chi_{S_{k,e}^-}\end{equation}
which are supported in $S_{2k,e}$. \\
Observe that $w^1_e$ and $w^2_e$ are both nonnegative in $S_{2k,e}$, have disjoint supports and satisfy $w_e^i=0$ on $\partial S_{k,e}^\pm$ and also on $\partial S_{2k,e}$. Moreover 
$w_e^i(\sigma_e(x))=w_e^i(x)$ for $x\in S_{k,e}^+$.\\
Further, since $w_e$ satisfies $L_ew_e=0$ in $S_{2k,e}$ and $w_e=0$ on $\partial S_{2k,e}$, multiplying by $w_e^+\chi_{S_{k,e}^+}+w_e^- \chi_{S_{k,e}^-}$ and integrating over $S_{2k,e}$ we obtain
\[
\begin{split}
0=&\int_{S_{2k,e}}\nabla w_e\nabla \left(w_e^+\chi_{S_{k,e}^+}+w_e^- \chi_{S_{k,e}^-}\right)- V_e(x)w_e\left(w_e^+\chi_{S_{k,e}^+}+w_e^- \chi_{S_{k,e}^-}\right)\\
=& \int_{S_{2k,e}}|\nabla (w_e^+\chi_{S_{k,e}^+})|^2+|\nabla (w_e^- \chi_{S_{k,e}^-})|^2-V_e(x)\left( (w_e^+\chi_{S_{k,e}^+})^2+(
w_e^- \chi_{S_{k,e}^-})^2\right)\\
=&\int_{S_{2k,e}}|\nabla w_e^1|^2-V_e(x)(w_e^1)^2	.
\end{split}\]
Then, \eqref{eq:confronto-V}, together with $w_e^i(\sigma_e(x))=w_e^i(x)$,
gives
\begin{align*}
&Q_{u,S_{2k,e}}(w_e^1,w_e^1)=Q_{es,S_{2k,e}}(w_e^1,w_e^1)=\int_{S_{2k,e}}|\nabla w_e^1|^2-V_{es}(x)(w_e^1)^2\\
&\le \int_{S_{2k,e}}|\nabla w_e^1|^2-V_{e}(x)(w_e^1)^2=0
\end{align*}
where we are denoting by $Q_{es,D}$ the quadratic form associated with the operator $L_{es}$ in $D$.
Similarly we can show
\[Q_{u,S_{2k,e}}(w_e^2,w_e^2)\leq0.\]
Let us denote by $\tilde w_e^i$ the extension of $w_e^i$ in all $\Omega$ by the invariance by rotations of angle $\frac {2\pi}k$.
For every $e\in \mathcal S$ we can define the function
\[\xi_e:=\left(\frac {\int _{  S_{2k,e}} w_{e}^2\varphi_1  }{\int _{  S_{2k,e}} w_{e}^1\varphi_1  }\right)^{\frac 12} w_e^1-\left(\frac {\int _{  S_{2k,e}} w_{e}^1\varphi_1}{\int _{  S_{2k,e}} w_{e}^2\varphi_1  } \right)^{\frac 12} w_{e}^2=A_e w_e^1-B_e w_{e}^2\]
which is supported in $S_{2k,e}$ and denote by $\tilde \xi_e$ its extension in all of $\Omega$ by the invariance  
by rotations of angle $\frac {2\pi} k$. Then, by construction
\[\int_{\Omega} \tilde  \xi_e\varphi_1=k   \int _{S_{2k,e}}\xi_e\varphi_1=0\]
and
\begin{equation}\label{eq:forma-negativa}
Q_{u,\Omega}(\tilde \xi_e,\tilde \xi_e)=k
Q_{u,S_{2k,e}}(\xi_e,\xi_e)=kA_e^2Q_{u,S_{2k,e}}(w_e^1,w_e^1)+kB_e^2Q_{u,S_{2k,e}}(w_e^2,w_e^2)
\leq 0\end{equation}
since $w_e^1$ and $w_e^2$ have disjoint supports. Assume $e=e_\psi=(\cos \psi,\sin \psi)$, with $\psi\in [0,2\pi)$. 
Let us consider the map
\begin{equation}\label{eq:h-psi}
h: [0,2\pi)\to \R \ \ , \ \ h(\psi):=\int _{S_{2k,\psi}} \xi_\psi \varphi_2 \ dx \end{equation}
which is continuous with respect to the angle $\psi$.  Here $\xi_\psi$ stands for $\xi_{e_\psi}$.\\
We want to prove that there exists $\psi\in [0,\frac \pi k)$ such that $h(\psi)=0$. 
If $h(0)=0$ then we are done.
If else $h(0)\neq 0$ as in the proof of Proposition \ref{prop-3.2} we have $S_{k,\frac \pi k}^-=S_{k,0}^+$
and $S_{k,\frac \pi k}^+= R_{\frac {2\pi}k}(S_{k,0}^-)$.
Observe that, for $x\in S_{k,\frac \pi k}^-$,  
\[w_{\frac \pi k}(x)=u(\sigma_{\frac \pi k}(x))-u(x) =u(R_{\frac {2\pi}k}(\sigma_0(x)))-u(x)=u(\sigma_0(x))-u(x)=
w_0(x)\]
since $u$ is $k$-invariant, while, for $x\in S_{k,\frac \pi k}^+$
\[w_{\frac \pi k}(x)=u(\sigma_{\frac \pi k}(x))-u(x) =u(\sigma_{0}(y))-u( R_{\frac {2\pi}k}( y))=w_0(y)\]
for $y\in S_{k,0}^-$.
From this it follows that 
\[
\begin{split}
\tilde w_{\frac \pi k}^1(x)=w_{\frac \pi k}^+\chi_{S_{k,\frac \pi k}^+}-w_{\frac \pi k}^-\chi_{S_{k,\frac \pi k}^-}
=\tilde w_{0}^2
\end{split}\]
and 
\[\tilde w_{\frac \pi k}^2=\tilde w_{0}^1 . \]
This implies that
\[
A_{e_{\frac \pi k}}=\left(\frac {\int _{  S_{2k,\frac \pi k }} w_{\frac \pi k}^2\varphi_1  }{\int _{  S_{2k,\frac \pi k}} w_{\frac \pi k}^1\varphi_1  }\right)^{\frac 12} =\left(\frac {\int _{  S_{2k,0}} w_{\psi}^1\varphi_1  }{\int _{  S_{2k,0}} w_{\psi}^2\varphi_1  }\right)^{\frac 12} =B_{e_{0}}
\]
where we used the $k$-invariance of $\varphi_1$, and, in the same manner 
\[B_{e_{\frac \pi k}}=A_{e_{0}}.\]
It is then easy to see that $\tilde \xi_{\frac \pi k}=-\tilde \xi_0$, that implies, together with $\varphi_1(R_{\frac {2\pi}k}(x))=\varphi_1(x)$, that 
$h(\frac \pi k)=-h(0)$.

Since $h$ is a continuous function then it has one zero in $(0,\frac \pi k)$. This means that there exists a direction $e'\in \mathcal S$ such that
$\tilde \xi_{e'}$ is orthogonal to $\varphi_2$ in $L^2(\Omega)$, since $\int_\Omega \tilde \xi_{e'}\varphi_2 =k\int_{S_{2k,e'}} \xi_{e'}\varphi_2 =0$. 
Since $\varphi_1, \varphi_2$ and  $\tilde \xi_{e'}$ are $k$-invariant and  $m_k(u)\leq 2$ then 
\[Q_{u,\Omega}(\tilde \xi_{e'},\tilde \xi_{e'})\geq 0\]
which, together with \eqref{eq:forma-negativa} implies that 
$\tilde\xi_{e'}$ is a minimizer for $Q_{u,\Omega}$ in the space of functions $k$-invariant and hence, by the principle of symmetric criticality of \cite{Palais}, $\tilde\xi_{e'}$ weakly solves
$L_u\tilde \xi_{e'}=0$ in $\Omega$ and also $L_u \xi_{e'}=0$ in $S_{2k,e'}$. Recall that $\xi_{e'}=0$ on $\partial S_{2k,e'}$ and since $\xi_{e'}(\sigma_{e'}(x))=\xi_{e'}(x)$ for any $x\in S_{2k,e'}$ (by the symmetry of $w_{e'}^i$), then $\frac {\partial \xi_{e'}}{\partial \nu}=0$ on $\Gamma_{2,e'}$, where $\frac {\partial}{\partial \nu}$ denotes the inner normal derivative on $\Gamma_{2,e'}$. 
But, this implies that the function
\begin{equation}\label{hat-xi}
\widehat \xi_{e'}(x)=
\begin{cases}
\xi_{e'} (x) & \hbox{ for } x\in S_{k,e'}^+\\
0& \hbox{ for } x\in S_{k,e'}^-
\end{cases}\end{equation}
is also a weak solution to $L_u\widehat \xi_e=0$ in $S_{2k,e'}$, contradicting the unique continuation principle for this equation.
The only possibility is that $\xi_{e'}\equiv 0$ which is not possible since we are assuming that $w_{e'}$ changes sign in $S_{k,e'}^+$.

\

\noindent {\bf Case 2: }$S_*\neq \emptyset$.\\
Let us assume that $e=e_\psi\in S_*$, so that $w_\psi\equiv 0$ in $S_{k,\psi}^+$. By the symmetry of $u$ we also have that
$\varphi_1(\sigma_\psi(x))=\varphi_1(x)$.\\ 
Let us denote by $g_\psi$ the unique first positive eigenfunction of $L_u$ in $S_{k,\psi}^+$ with Dirichlet boundary conditions, which corresponds to $\l_1(L_u, S_{k,\psi}^+)$ normalized such that $\int_{S_{k,\psi}^+}g_\psi^2=1$, by 
$\widehat g_\psi$ its odd extension on $S_{2k,\psi}$, and by $\tilde g_\psi$ its extension in $\Omega$ by the invariance by rotations of angle $\frac{2\pi}k$.\\
It is easy to see that
\[\begin{split}
&\int_\Omega \tilde g_\psi\varphi_1\ dx=k\int_{S_{2k,\psi}} \widehat g_\psi\varphi_1\ dx=k\left(\int_{S_{k,\psi}^+}g_\psi\varphi_1\ dx+\int_{S_{k,\psi}^-}(-g_\psi(\sigma_\psi(x)))   \varphi_1\ dx\right)\\
&=k\left(\int_{S_{k,\psi}^+}g_\psi\varphi_1\ dx-\int_{S_{k,\psi}^+}g_\psi(x)   \varphi_1(\sigma_\psi(x))\ dx\right)\\
&=k\left(\int_{S_{k,\psi}^+}g_\psi\varphi_1\ dx-\int_{S_{k,\psi}^+}g_\psi(x)   \varphi_1(x)\ dx\right)=
0.
\end{split}\]
Moreover, by the symmetry of $u$ in $S_{k,\psi}^+$ we have
\[Q_{u,\Omega}(\tilde g_\psi,\tilde g_\psi)=k Q_{u, S_{2k,\psi}}(\widehat g_\psi,\widehat g_\psi)=2k Q_{u, S_{k,\psi}^+}( g_\psi, g_\psi)=2k \l_1(L_u, S_{k,\psi}^+)<0
\]
since we are assuming that $e_\psi\in S_*$.\\
To conclude we want to show that, when we consider the direction $e_{\psi+\frac \pi{2k}}$ (which divides $S_{k,\psi}^+$ into two equal parts), then we must have either that $w_{\psi+\frac \pi{2k}}>0$ in $S_{k,\psi+\frac \pi{2k}}^+$ either $w_{\psi+\frac \pi{2k}}<0$ in $S_{k,\psi+\frac \pi{2k}}^+$ or, else, that $w_{\psi+\frac \pi{2k}}\equiv 0$ in $S_{k,\psi+\frac \pi{2k}}^+$ and $\l_1(L_u, S_{k,\psi+\frac \pi{2k}}^+)\geq 0$.\\
i) First we show that whenever $w_{\psi+\frac \pi{2k}}\equiv 0$ in $S_{k,\psi+\frac \pi{2k}}^+$ then $\l_1(L_u, S_{k,\psi+\frac \pi{2k}}^+)\geq 0$.\\
Suppose by contradiction that $\l_1(L_u, S_{k,\psi+\frac \pi{2k}}^+)< 0$. Since $u$ is symmetric with respect to $r_{\psi+\frac \pi{2k}}$, namely $u(x)=u(\sigma_{\psi+\frac \pi{2k}})$ in $S_{k,\psi+\frac \pi{2k}}^+$, and since $r_{\psi+\frac \pi{2k}}$ divides the sector $S_{k,\psi}^+$ into two equal parts, then the first eigenfunction $g_\psi$ in $S_{k,\psi}^+$ is symmetric with respect $r_{\psi+\frac \pi{2k}}$. \\
Next we let $g_{\psi+\frac \pi{2k}}$ be the first positive eigenfunction of the linear operator $L_u$ in the sector $S_{k,\psi+\frac \pi{2k}}^+$
with Dirichlet boundary conditions, which corresponds to $\l_1(L_u, S_{k,\psi+\frac \pi{2k} }^+)$ normalized such that $\int_{S_{k,\psi+\frac \pi{2k}}^+}g_{\psi+\frac \pi{2k}}^2=1$, by 
$\widehat g_{\psi+\frac \pi{2k}}$ its odd extension on $S_{2k,\psi+\frac \pi{2k}}$, and by $\tilde g_{\psi+\frac \pi{2k}}$ its extension in $\Omega$ by the invariance by rotations of angle $\frac{2\pi}k$. Again by the symmetry of $u$ we have that $g_{\psi+\frac \pi{2k}}(\sigma_{\psi+\frac \pi{k}}(x))=g_{\psi+\frac \pi{2k}}(x)$ in $S_{k, \psi+\frac {2\pi}k} ^+$.
Then
\[
\begin{split}
&\int_{S_{2k,\psi+\frac \pi{2k}}}\widehat g_{\psi+\frac \pi{2k}} \tilde g_\psi\ dx=\int_{S_{k,\psi+\frac \pi{2k}}^+} g_{\psi+\frac \pi{2k}} \tilde g_\psi\ dx+\int_{S_{k,\psi+\frac \pi{2k}}^-}\widehat g_{\psi+\frac \pi{2k}} \widehat g_\psi\ dx
\end{split}\]
where 
\[\int_{S_{k,\psi+\frac \pi{2k}}^-}\widehat g_{\psi+\frac \pi{2k}} \widehat g_\psi\ dx=0\]
since $g_\psi$ is odd with respect $r_\psi$ while $\widehat g_{\psi+\frac \pi{2k}}$ is even with respect to $r_\psi$. Next
\[\begin{split}
&\int_{S_{k,\psi+\frac \pi{2k}}^+} g_{\psi+\frac \pi{2k}} \tilde g_\psi\ dx=\int_{(S_{k,\psi+\frac \pi{2k}}^+)\cap S_{k,\psi}^+} g_{\psi+\frac \pi{2k}} g_\psi\ dx+\int_{(S_{k,\psi+\frac \pi{2k}}^+)\setminus S_{k,\psi}^+} g_{\psi+\frac \pi{2k}} \tilde g_\psi\ dx\\
& =\int_{(S_{k,\psi+\frac \pi{2k}}^+)\cap S_{k,\psi}^+} g_{\psi+\frac \pi{2k}} g_\psi\ dx+\int_{(S_{k,\psi+\frac \pi{2k}}^+)\cap S_{k,\psi}^+} g_{\psi+\frac \pi{2k}}(\sigma_{\psi+\frac \pi k}(x)) \tilde g_\psi  (\sigma_{\psi+\frac \pi k}(x)) \ dx\\
& = \int_{(S_{k,\psi+\frac \pi{2k}}^+)\cap S_{k,\psi}^+} g_{\psi+\frac \pi{2k}} g_\psi\ dx+\int_{(S_{k,\psi+\frac \pi{2k}}^+)\cap S_{k,\psi}^+} g_{\psi+\frac \pi{2k}}(x) (- g_\psi  (x)) \ dx=0
\end{split}
\]
since $\tilde g_\psi(x)=\widehat g(R_{-\frac {2\pi}k}(x))=-g_\psi(\sigma_{\psi+\frac \pi k}(x))$ for $x\in S_{k,\psi+\frac \pi{2k}}^+\setminus S_{k,\psi}^+$.
Similarly
\[\int_{S_{2k,\psi+\frac \pi{2k}}}\widehat g_{\psi+\frac \pi{2k}} \varphi_1\ dx=0\]
so that $\tilde g_{\psi+\frac \pi{2k}}$ is orthogonal to $\varphi_1$ and to $\tilde g_\psi$ in $L^2(\Omega)$. By assumption then we have 
\[\begin{split}
&Q_{u,\Omega}(\tilde g_{\psi+\frac \pi{2k}},\tilde g_{\psi+\frac \pi{2k}})=k Q_{u,S_{2k, \psi+\frac \pi{2k}}}(\widehat g_{\psi+\frac \pi{2k}},\widehat  g_{\psi+\frac \pi{2k}})\\
&=2k Q_{u,S_{k, \psi+\frac \pi{2k}}^+}(g_{\psi+\frac \pi{2k}},g_{\psi+\frac \pi{2k}})=2k\l_1(L_u, S_{k,\psi+\frac \pi{2k}}^+)< 0
\end{split}
\]
which contradicts the bound on the $k$-Morse index. This shows that whenever $w_{\psi+\frac \pi {2k}} \equiv 0$ in $S_{k, \psi+\frac \pi {2k}}^+$, then $\l_1(L_u, S_{k, \psi+\frac \pi {2k}}^+)\geq 0$ and concludes this part. \\

\

ii) It lasts to prove that whenever $w_{\psi+\frac \pi {2k}}\neq 0$ 
in $S_{k, \psi+\frac \pi {2k}}^+$, then  the function $w_{\psi+\frac \pi {2k}}$ cannot change sign in $S_{k, \psi+\frac \pi {2k}}^+$.
Assume by contradiction that it changes sign. We can then define the functions $w_e^1$ and $w_e^2$ as in \eqref{eq:w1e2} relative to the direction $e=e_{\psi+\frac \pi{2k}}$, which are nonnegative, symmetric with respect to $r_{\psi+\frac \pi{2k}}$ and satisfies 
\[Q_{u,S_{2k,\psi+\frac \pi{2k}}}(w_e^i,w_e^i)\leq 0.\]
Moreover, by the symmetry of $u$ we have that $w_{\psi+\frac \pi{2k}}(\sigma_\psi(x))=w_{\psi+\frac \pi{2k}}(x)$ which implies that 
$w_e^i(\sigma_\psi(x))=w_e^i(x)$. Next we construct as before the function $\xi_e$, and $\tilde \xi_e$ which, again satisfies by construction $\xi_e(\sigma_\psi(x))=\xi_e(x)$ and also
\begin{equation}\label{f-neg}Q_{u,\Omega}(\tilde \xi_e,\tilde\xi_e)\leq 0.\end{equation}
Further $\tilde \xi_e$ is orthogonal to $\varphi_1$ by construction and $\tilde \xi_e$ is orthogonal to $\tilde g_\psi$ since the first is even with respect to $r_\psi$ while the second is odd. 
Since $\varphi_1$ and $\tilde g_\psi$ and  $\tilde \xi_{e}$ are $k$-invariant and  $m_k(u)\leq 2$ then 
\[Q_{u,\Omega}(\tilde \xi_{e},\tilde \xi_{e})\geq 0\]
which, together with \eqref{f-neg} implies that 
$\tilde\xi_{e}$ is a minimizer for $Q_{u,\Omega}$ in the space of $k$-invariant functions, and, as before, $\tilde\xi_{e}$ solves
$L_u\tilde \xi_{e}=0$ in $\Omega$ and also $L_u \xi_{e}=0$ in $S_{2k,e}$. Then we can conclude defining the function $\widehat \xi$ as in \eqref{hat-xi}
as in the end of the proof of case 1 showing that it is not possible that $w_{\psi+\frac \pi{2k}}$ changes sign in $S_{k,\psi+\frac \pi{2k}}$. 
\end{proof}

\

Of course Remark \ref{rem:1} holds also in this case.

\section{Unbounded domains}
In this section we extend the previous results to the case of an unbounded radial domain $\Omega$, where $\Omega$ either coincides with $\R^2$ or $\Omega=\R^2\setminus B$, extending some ideas in \cite{GPW}
to the case of functions which are $k$-invariant.\\
As before we denote by $C^1_{0,k}$ the subset of $C^1_0(\Omega)$ of functions that satisfy \eqref{k-inv}.\\
To prove the results we only highlight the difference between the bounded case. First we can prove an extension of Proposition \ref{prop-2.2}, namely
\begin{proposition}\label{prop-5.1}
Assume $u$ is a solution to \eqref{P} and possibly \eqref{eq:bc} that satisfies \eqref{k-inv} and such that 
$|\nabla u|\in L^2(\Omega)$.
Suppose further that there exists $e\in \mathcal S$ such that $w_e\equiv 0$ in $S_{k,e}^+$ and
\[\inf_{\phi\in C^1_0(S_{k,e}^+)}Q_{u, S_{k,e}^+}(\phi,\phi)\geq 0.\]
Then or $u$ is radial or it is strictly monotone in the angular variable in $S_{k,e}^+$.
\end{proposition}
\begin{proof}
As in the proof of Proposition \ref{prop-2.2} we have that the angular derivative $u_\theta$ satisfies
\[
\begin{cases}
-\Delta u_\theta=f'(|x|,u)u_\theta & \hbox{ in }S_{k,e}^+\\
u_\theta=0 & \hbox{ on }\partial S_{k,e}^+
\end{cases}\]
We claim that either $u_\theta\equiv 0$ in $S_{k,e}^+$ showing that $u$ is radial or, else, $u_\theta$ has a sign in $S_{k,e}^+$.\\
Using a radial cut-off function $\xi_R$ supported in $\Omega\cap B_R(0)$ and reasoning exactly as in the proof of Proposition 2.5 of \cite{GPW} we can prove that in this case the function $u_\theta^+$ satisfies
\[\int _{S_{k,e}^+}\left(\nabla u_\theta^+\nabla \phi-f'(|x|,u)u_\theta^+\phi\right)\, dx=0\]
for any $\phi\in C^{\infty}_0(S_{k,e}^+)$, meaning that $u_\theta^+$, if it is nonzero, solves
\[-\Delta u_\theta^+=f'(|x|,u)u_\theta^+ \ \ \hbox{ in }S_{k,e}^+\]
and by the unique continuation principle or $u_\theta^+\equiv 0$ or $u_\theta^+>0$ in $S_{k,e}^+$ concluding the proof. 
\end{proof}

By the symmetry assumption on $u$, whenever $u$ is nonradial then it is strictly monotone in the polar angle $\theta$ also in the sector $S_{k,e}^-$.

\

Moreover, when $f$ or $f'$ are convex in the second variable we can extend also Proposition \ref{prop-2.1}, getting the following result:
\begin{proposition}\label{prop-5.2}
Assume $u$ is a solution to \eqref{P} and possibly \eqref{eq:bc} that satisfies \eqref{k-inv} and such that 
$|\nabla u|\in L^2(\Omega)$ and $m_k(u)<\infty$.
Suppose furthermore that  $f$ or $f'$ are convex in the second variable and that there exists a direction $e\in \mathcal S$ s.t. 
\begin{equation}\nonumber
w_e>0 \ \ \hbox{ in }S_{k,e}^+ \ \ \hbox{ or }\ w_e<0 \ \ \hbox{ in }S_{k,e}^+.
\end{equation}
Then there exists another direction $e'\in \mathcal S$ s.t. $w_e\equiv 0$ and
\begin{equation}\label{eq:forma-positiva}
\inf _{\phi\in C^1_0( S_{k,e}^+)}Q_{u, S_{k,e}^+}(\phi,\phi)\geq0.
\end{equation}

\end{proposition}
\begin{proof}
We define $\tilde \psi$ as in \eqref{tilde-psi} and we want to prove that $w_{\tilde\psi}\equiv 0$ in $S_{k,\tilde \psi}^+$. Indeed \eqref{eq:forma-positiva} follows as in Proposition \ref{prop-2.1} using Lemma 2.1 in \cite{GPW}. The proof follows very closely the one in Proposition 2.8 of \cite{GPW}, adapted to the case of the sector $S_k$ as in  Proposition \ref{prop-2.1}.\\
Assume by contradiction that $w_{\tilde \psi}\neq 0$ in $S_{k,\tilde \psi}^+$. Then necessarily  $w_{\tilde \psi}> 0$ in $S_{k,\tilde \psi}^+$ by the strong maximum principle. Moreover applying the Hopf Lemma on the straight lines $\Gamma_{2,\tilde \psi}$ and $\Gamma_{3,\tilde\psi}$, where $w_{\tilde \psi}$ vanishes we have:
\begin{equation}\label{eq:primo-passo}
\frac{\partial w_{\tilde \psi}}{\partial \theta}(x)>0 \text{ on }\Gamma_{2,\tilde \psi}\setminus\{0\}\ \ ,\ \ \frac{\partial w_{\tilde \psi}}{\partial \theta}(x)<0 \text{ on }\Gamma_{3,\tilde \psi}\setminus\{0\}
\end{equation}
where $\theta$ denotes the angular variable in polar coordinates. Since by hypothesis $m_k(u)<\infty$, there exists $R_0>0$ such that 
$Q_{u,\Omega \setminus B_R} (\phi,\phi)\ge 0$ for every $\phi\in C^1_0(\Omega \setminus B_R) $ that satisfies \eqref{k-inv} and every $R>R_0$. Then also $Q_{u,S_{2k,\tilde \psi}\setminus B_R} (\phi,\phi)\ge 0$ for every $ \phi\in C^1_0(S_{2k,\tilde \psi} \setminus B_R) $ and every $R>R_0$
which implies that
$Q_{u,S_{k,\tilde \psi}^+\setminus B_R} (\phi,\phi)\ge 0$ for every $ \phi\in C^1_0(S_{k,\tilde \psi}^+ \setminus B_R) $ and every $R>R_0$.\\
We fix $R_1>R_0$ and we claim that there exists $\e_0>0$ such that 
\begin{equation}\label{eq:passo}
w_{\tilde \psi+\e}\geq 0 \ \ \text{ in }B_{R_1}\cap S_{k,\tilde\psi+\e}^+\ \ \forall \e\in [0,\e_0).
\end{equation}
In the case $\Omega=\R^2\setminus B$ let $B_\delta$ be a neighborhood of $\partial B$ in $\Omega$ of small measure, in case of $\Omega=\R^2$ instead let $B_\delta$ be a neighborhood of the origin of small measure. The measure of $B_\delta$ in both cases is so small to allow the strong maximum principle
to hold in $B_\delta\cap S_{k,\tilde \psi+\e}^+$ for the operator $L_{\tilde \psi+\e}$ for sufficiently small $\e>0$. We first show that
\begin{equation}\label{eq:secondo-passo}
w_{\tilde \psi+\e}(x)\geq  0 \ \text{ in }
B_{R_1}\cap \left(S_{k,\tilde\psi+\e}^+ \setminus B_\delta\right) \ \ \text{ for all } \e \in [0,\e_0)\end{equation}
If \eqref{eq:secondo-passo} is not true we have sequences $\e_n\to 0$,
$x_n \in (B_{R_1}\cap S_{k,\tilde{\psi}+\e_n}^+)\setminus B_{\delta}$ such that
$w_{\tilde{\psi}+\e_n}(x_n)
<0$. After passing to a subsequence,
$x_n\to x_0\in \overline{(B_{R_1}\cap 
  S_{k,\tilde{\psi}}^+)\setminus B_{\delta}}$ and
$w_{\tilde{\psi}}(x_0)=0$, hence  $x_0\in (\Gamma_{2,\tilde \psi}\cup \Gamma_{3,\tilde \psi})\cap (\overline{B_{R_1}\setminus B_\delta})$. Assume, without l.o.g. $x_0\in \Gamma_{2,\tilde \psi}$.
Since $w_{\tilde{\psi}+\e_n}(x)=0$ on
$\Gamma_{2,\tilde{\psi}+\e_n}$ and $w_{\tilde{\psi}+\e_n}(x_n)<0$, there should be points
$\xi_n$ on the line segment joining $x_n$ with $\Gamma_{2,\tilde\psi+\e_n}$ and
perpendicular to $ \Gamma_{2,\tilde\psi+\e_n}$,  such that
$\frac{\de w_{\tilde{\psi}+\e_n}}{\de \theta}(\xi_n)<0$.
Passing to the limit we get $\frac{\de w_{\tilde{\psi}}}
{\de \theta}(x_0)\le 0$ in contradiction with
(\ref{eq:primo-passo}). So we get (\ref{eq:secondo-passo}).\\
By the strong maximum principle, the definition of $B_{\delta}$ and
(\ref{eq:secondo-passo}) we get $w_{\tilde{\theta}+\e_n}\leq 0$ also in $
S_{k,\tilde{\psi}+\e}^+\cap  B_{\delta}$ and hence (\ref{eq:passo})
holds.\\
It lasts to prove that $w_{\tilde{\psi}+\e_n}\geq 0$ also in $S_{k,\tilde{\psi}+\e_n}^+\setminus B_{R_1}$ for all $\e\in [0,\e_0)$.
We recall that $w_{\tilde{\psi}+\e}$ satisfies
\begin{equation}\label{2.15}
\begin{array}{ll}
-\Delta w_{\tilde{\psi}+\epsilon}-V_{\tilde{\psi}+\epsilon}(x)w_{\tilde{\psi}+\epsilon}=0 & \hbox{ in }
S_{k,\tilde{\psi}+\e}^+.
\end{array}
\end{equation}
By (\ref{eq:passo}) the function
$v:=w_{\tilde{\psi} +\epsilon}^-\: \chi_{\stackrel{}{S_{k, \tilde{\psi}+\e}^+}}$ has its support strictly contained
in $S_{k,\tilde{\psi}+\e}^+\setminus B_{R_0}$. 
We claim that
\begin{equation}
  \label{eq:equiv0}
v \equiv 0.
\end{equation}
We first consider the case where $f$ is convex in the second variable,
and we let $\phi \in
C_0^\infty(S_{k,\tilde{\psi}+\e}^+\setminus B_{R_0})$. By previous considerations
$Q_{u,S_{k,\tilde{\psi}+\e}^+\setminus B_{R_0}} $ defines a (semidefinite) scalar product on
$C^1_0(S_{k,\tilde{\psi}+\e}^+\setminus B_{R_0}$), and the corresponding
Cauchy-Schwarz-inequality yields
$$
\left(
\int_{S_{k,\tilde{\psi}+\e}^+\setminus B_{R_0}}
\nabla \psi\nabla \rho -f'(|x|,u) \psi\rho\right)^2\le Q_{u,S_{k,\tilde{\psi}+\e}^+\setminus B_{R_0}}(\psi,\psi) Q_{u,S_{k,\tilde{\psi}+\e}\setminus B_{R_0}}(\rho,\rho)$$ 
for all $\psi,\rho \in C^1_0(S_{k,\tilde{\psi}+\e}^+\setminus B_{R_0})$.
Consequently, we obtain
\[
  \left(
\int_{S_{k,\tilde{\psi}+\e}^+\setminus B_{R_0}}
\nabla (v \xi_R)
\nabla \rho -f'(|x|,u) (v \xi_R)\rho\right)^2 \le Q_{u,S_{k,\tilde{\psi}+\e}^+\setminus B_{R_0}}(v \xi_R, v\xi_R)\: Q_{u,S_{k,\tilde{\psi}+\e}^+\setminus B_{R_0}} (\phi,\phi) 
\]
for $R>0$, where $\xi_R$ is a cut-off function supported in $B_{2R}$. 
By Lemma~2.3 (ii) in \cite{GPW}, we have
$$
\limsup_{R \to \infty} Q_{u, S_{k,\tilde{\psi}+\e}^+\setminus B_{R_0}}(v \xi_R, v \xi_R) \le 0
$$
so that
$$
\int_{S_{k,\tilde{\psi}+\e}^+\setminus B_{R_0}}\left(\nabla v \nabla \phi -f'(|x|,u) v
\phi \right)\:dx =
 \lim_{R \to \infty}
 \int_{S_{k,\tilde{\psi}+\e}^+\setminus B_{R_0}}
\nabla (v \xi_R)
\nabla \phi -f'(|x|,u) (v \xi_R)\phi
= 0.
$$
Since $\phi \in C_0^\infty(S_{k,\tilde{\psi}+\e}^+\setminus B_{R_0})$ was chosen
arbitrarily, we conclude that $v$ is a solution of
$$
-L_uv= 0 \qquad \text{in $S_{k,\tilde{\psi}+\e}^+\setminus B_{R_0}$.}
$$
Then however $v=w_{\tilde{\psi} +\epsilon}^- \chi_{S_{k,\tilde{\psi} +\epsilon} ^+}\equiv 0$ by the unique continuation
principle, since $w_{\tilde{\psi}+\epsilon}
^-(x) \equiv 0$ in $B_{R_1} \cap S
_{k,\tilde{\psi}+\e}^+$ by (\ref{eq:passo}) and $R_1>R_0$. Hence \eqref{eq:equiv0} holds.

\

Next we consider the case where $f'$ is convex in the second
variable. Since every function $\tau \in
C^1_0(S_{k,\tilde{\psi}+\e}^+\setminus B_{R_0})$ can be extended to an
odd function $\tilde{\tau} \in C^1_0(S_{2k,\tilde{\psi}+\e} \setminus B_{R_0})$ with
respect to the
reflection at $r_{\tilde{\psi} +\epsilon}$, we have 
$$
\int_{S_{k,\tilde{\psi}+\e}^+\setminus B_{R_0}}|\nabla \tau|^2-V_{\tilde \psi+\e}(x) \tau^2
\geq \int_{S_{k,\tilde{\psi}+\e}^+\setminus B_{R_0}}|\nabla \tau|^2-V_{es}(x) \tau^2
=\frac 12 Q_{u,S_{2k,\tilde{\psi}+\e}} (\tilde{\tau},\tilde{\tau})\geq 0 
$$
for all $\tau \in C^1_0(S_{k,\tilde{\psi}+\e}^+\setminus B_{R_0})$.
Hence $Q_{{\tilde\psi+\e}}(\tau,\tau):=\int_{S_{k,\tilde{\psi}+\e}^+\setminus B_{R_0}}|\nabla \tau|^2-V_{\tilde \psi+\e}(x) \tau^2$ defines a (semidefinite) scalar product on
$C^1_0(S_{k,\tilde{\psi}+\e}^+ \setminus B_{R_0})$, and the corresponding
Cauchy-Schwarz-inequality reads
\begin{equation}
  \label{eq:chauchy-schwarz-2}
\left(Q_{{\tilde{\psi}+\e}}(\psi,\rho)\right)^2 \le Q_{{\tilde{\psi}+\e}}(\psi,\psi)
Q_{{\tilde{\psi}t+\e}}(\rho,\rho),
\end{equation}
which by density holds for  all $\psi,\rho \in H^1_0(S_{k,e_{\tilde\psi+\e}}^+ \setminus B_{R_0})$ vanishing a.e. outside a bounded set. Now we let again $\phi \in
C_0^\infty(S_{k,\tilde{\psi}+\e}^+\setminus B_{R_0})$. By Lemma~2.3(i) in \cite{GPW}, we have
$$
\limsup_{R \to \infty} Q_{{\tilde\psi+\e}}(v \xi_R, v \xi_R) \le 0.
$$
Combining this with \eqref{eq:chauchy-schwarz-2}, we find  that
$$
\int_{S_{k,\tilde{\psi}+\e}^+\setminus B_{R_0}}\Bigl(\nabla v \nabla \phi -V_{{\tilde\psi+\e}}(x) v
\phi \Bigr)\:dx = \lim_{R \to \infty}Q_{{\tilde\psi+\e}}(v \xi_R,\phi)= 0
$$
Since $\phi \in C_0^\infty(\Sigma_{\tilde{\psi}+\e}\setminus B_{R_0})$ was chosen
arbitrarily, we conclude that $v$ is a solution of
$$
-\Delta v -V_{{\tilde\psi+\e}}(x) v= 0 \qquad \text{in $S_{k,\tilde{\psi}+\e}^+\setminus B_{R_0}$.}
$$
As above, this implies \eqref{eq:equiv0} by the unique continuation principle.\\
As a consequence of \eqref{eq:equiv0}, we have got $w_{\tilde\psi+\e}\ge 0$ in $S_{k,\tilde\psi+\e}^+$ contradicting the definition of $\tilde \psi$. Then the definition of $\tilde\psi$ implies that
  $w_{\tilde\psi}\equiv 0$.  

\end{proof}

\

In the case where $f$ is convex in the second variable then we have:
\begin{proposition}\label{prop-5.3}
Assume $u$ is a solution to \eqref{P} and possibly \eqref{eq:bc} that satisfies \eqref{k-inv} and such that 
$|\nabla u|\in L^2(\Omega)$ and $m_k(u)\leq 2$.
Suppose furthermore that  $f$ is convex in the second variable. Then, or $u$ is radial or, else, there exists a direction $e\in \mathcal S$ such that $u$ is symmetric with respect to $r_e$ in $S_{2k,e}$ and it is strictly monotone in the angular variable in the sectors $S_{k,e}^+$ and $S_{k,e}^-$.
\end{proposition}
\begin{proof}
We want to prove that there exists a direction $e\in \mathcal S$ such that
\begin{equation}\label{eq:forma-positiva-passo}
\inf_{\psi\in C^1_0(S_{k,e}^+)}Q_{u,S_{k,e}^+}(\psi,\psi)\geq 0 \ \ \text{ or }\inf_{\psi\in C^1_0(S_{k,e}^-)}Q_{u,S_{k,e}^-}(\psi,\psi)\geq 0.
\end{equation}
With respect to this direction then, either  $w_e\equiv 0$, and then the monotonicity of $u$ follows by Proposition \ref{prop-5.1},  or else it can be proved, as in the proof of Theorem 1.4 in \cite{GPW} and as in Proposition \ref{prop-5.2}, using the function $v_R=w_e^+\chi_{S_{k,e}^+}\xi_R$ for a cut off $\xi_R$, that $w_e$ has one sign in $S_{k,e}^+$ so that the monotonicity follows by Propositions \ref{prop-5.2} and \ref{prop-5.1}.\\
Suppose that \eqref{eq:forma-positiva-passo} does not hold. Then, by Lemma 2.9 of \cite{GPW} there exists $\bar R>0$ such that for any $e\in \mathcal S$ and for every $R>\bar R$ either $\l_1(L_u, S_{k,e}^+\cap B_R)<0$ or $\l_1(L_u, S_{k,e}^-\cap B_R)<0$.
By definition of $m_k(u)$, $2$ is the maximal dimension of a subspace $X=\mathit{span}\{\psi_1,\psi_2\}$ of
$C^1_{0,k}$ such that $Q_{u,\Omega}(\psi,\psi)<0$  for every $\psi\in X\setminus\{0\}$. \\
We take a ball $B_R$ with radius $R$ sufficiently large to contain the supports of $\psi_1$ and $\psi_2$ and also $B_{\bar R}$. We deduce that in $B_R\cap \Omega$ the operator $L_u$ admits precisely $2$ negative eigenvalues in $H^1_{0,k}$ and the third is nonnengative. Let us denote by $\varphi_1$ and $\varphi_2$ two eigenfunctions of $L_u$, orthogonal in $L^2(\Omega)$, corresponding to the negative eigenvalues in $\Omega\cap B_R$.\\
We then follow exactly the proof of Proposition \ref{prop-3.2} in $\Omega\cap B_R$ and we obtain that there exists a direction $e\in \mathcal S$ for which at least one among $\l_1(L_u, S_{k,e}^+\cap B_R)$ and $\l_1(L_u, S_{k,e}^-\cap B_R)$ is nonnegative obtaining a contradiction. This proves \eqref{eq:forma-positiva-passo} and concludes the proof.
\end{proof}

In the case when $f'$ is convex instead we have:
\begin{proposition}\label{prop-5.4}
Assume $u$ is a solution to \eqref{P} and possibly \eqref{eq:bc} that satisfies \eqref{k-inv} and such that 
$|\nabla u|\in L^2(\Omega)$ and $m_k(u)\leq 2$.
Suppose furthermore that  $f'$ is convex in the second variable. Then, or $u$ is radial or, else, there exists a direction $e\in \mathcal S$ such that $u$ is symmetric with respect to $r_e$ in $S_{2k,e}$ and it is strictly monotone in the angular variable in the sectors $S_{k,e}^+$ and $S_{k,e}^-$.
\end{proposition}
\begin{proof}
First we let $S_*$ as in \eqref{eq:defS}.  
Next, as in the proof of the previous Proposition, from $m_k(u)\leq 2$, we deduce that there exists $R_0>0$ such that $L_u$ admits exactly two negative eigenvalues in $\Omega \cap B_R$, for every $R\geq R_0$, with eigenfunctions $k$-invariant while the third eigenvalue is nonnengative. Let $\varphi_1$ and $\varphi_2$ two eigenfunctions $k$-invariant of $L_u$, orthogonal in $L^2(\Omega)$ 
corresponding to the negative eigenvalues in $\Omega\cap B_{R}$, for some $R\geq R_0$.
Hence $\varphi_1,\varphi_2\in H^1_{0,k}(B_{R})$ and, extending these functions by zero in $\Omega \setminus B_{R}$
\[Q_{u,\Omega}(\varphi_i,\varphi_i)<0 \ \text{ and }\ Q_{u,\Omega}(\varphi_i,\varphi_j)=0\]
where the second equality follows by the orthogonality of $\varphi_1$ and $\varphi_2$ in $L^2(B_{R})$ and the fact that
$\varphi_i\in H^1_{0,k}(B_{R})$ solves $-\Delta \varphi_i-f'(|x|,u)\varphi_i=\l_i\varphi_i$  in $\Omega\cap B_R$.\\
\noindent {\bf Case 1:} $S_*=\emptyset$\\
Assume by contradiction that for every $e\in\mathcal S$ the function $w_e$ changes sign in $S_{k,e}^+$. 
We define the functions $w_e^1$ and $w_e^2$ as in \eqref{eq:w1e2}. They belong to $H^1_{0,k,loc}(\Omega)$.
Next we need to change a little the proof of Proposition \ref{prop-4.1} in order to deal with the unboundedness  of $\Omega$ following the proof of Theorem 1.3 in  \cite{GPW}. For every $e\in \mathcal S$ we define the
function
 \begin{equation}\label{eq:xi-secondo}
 \xi_e:=Q_{u,S_{2k,e}}( w_e^2,\varphi_1  )w_e^1-Q_{u,S_{2k,e}}( w_e^1,\varphi_1  )w_e^2=A_ew_e^1-B_ew_e^2\end{equation}
 which is supported in $S_{2k,e}$, and denote by $\tilde \xi_e$ its extension in all of $\Omega$ by the invariance by rotations of angle $\frac{2\pi}k$.
 Then, by construction
\[\begin{split}
Q_{u,\Omega}(\tilde\xi_e,\varphi_1)&=kQ_{u,S_{2k,e}}(\xi_e,\varphi_1)=k
\left[A_eQ_{u,S_{2k,e}}(w_e^1,\varphi_1)-B_eQ_{u,S_{2k,e}}(w_e^2,\varphi_1)\right]=0
\end{split}
\]
Observe that $Q_{u,\Omega}(\tilde\xi_e,\varphi_1)$ is well defined since $\varphi_1$ is supported in $B_{R_0}$ and $\xi_e$ and $\tilde \xi_e$ belong to $H^1_{0,loc}(S_{2k,e})$ and to $H^1_{0,k,loc}(\Omega)$ respectively. \\

Since $\varphi_2$ is supported in $B_{R}$ we can define the function $h(\psi)$ as in \eqref{eq:h-psi}. 
We then obtain, 
as in Proposition \ref{prop-4.1}, that there exists a direction $e'\in \mathcal S$ such that $h(\psi_{e'})=0$ showing that 
$\tilde \xi_{e'}$ is orthogonal to $\varphi_2$ in $L^2(\Omega)$, since $h(\psi_{e'})=\int_{\Omega}\tilde \xi_{e'}\varphi_2 =k\int_{S_{2k,e'}} \xi_{e'}\varphi_2 =0$. \\

From the definition of $k$-Morse index of $u$ it is easy to
deduce that $Q_{u,\Omega}(\rho,\rho) \ge 0$
for all $\rho \in C^1_0(\Omega)$ that satisfies \eqref{k-inv} which is orthogonal in $L^2(\Omega)$ to $\varphi_1$ and $\varphi_2$. By density this holds also for functions in $H^1_{0,loc}(\Omega)$ vanishing outside a compact set, and so we denote by $\mathcal H $ the subspace of functions in $H^1_{0,loc}(\Omega)$ vanishing outside a compact set that satisfies \eqref{k-inv} which are orthogonal in $L^2(\Omega)$ to $\varphi_1$ and $\varphi_2$.
Hence $Q_{u,\Omega}$ defines a semidefinite scalar product in
$\mathcal H$ with corresponding Cauchy Schwarz inequality
$$
Q_{u,\Omega}(\tau,\rho)^2 \le Q_{u,\Omega}(\tau,\tau) Q_{u,\Omega}(\rho,\rho) \qquad \text{for
  $\tau,\rho \in \mathcal H$.}
$$
We now fix $\phi \in C_0^\infty(\Omega)$ and put
$$
\phi_1= \phi - \sum_{i=1}^2 \frac{Q_{u,\Omega}(\phi,\varphi_i)}{Q_{u,\Omega}(\varphi_i,\varphi_i)}\varphi_i.
$$
Then $\phi_1 \in \mathcal H$ and, letting $\xi_R$ be a cut-off function supported in $B_R$, we have 
\begin{equation}
Q_{u,\Omega}(\tilde \xi_{e'} \xi_R,\phi)^2=Q_{u,\Omega}(\tilde \xi_{e'} \xi_R,\phi_1)^2 \le
Q_{u,\Omega}(\tilde \xi_{e'} \xi_R,\tilde\xi_{e'} \xi_R)\: Q_{u}(\phi_1,\phi_1) \quad \text{for $R>R_0$.}     \label{eq:cs-particular}
\end{equation}
Now, by Lemma 2.3 (i) of \cite{GPW}  we have that 
\begin{equation}
\label{eq:limsup}
\limsup_{R\to \infty} Q_{e',\Omega}(w_{e'}^j\xi_R,w_{e'}^j\xi_R)\leq 0
\end{equation} 
for $j=1,2$ and, since $w_{e'}^j(\sigma_{e'}(x))=w_{e'}^j(x)$ then by \eqref{eq:confronto-V}
\[Q_{u,S_{2k,e'}}(w_{e'}^j\xi_R, w_{e'}^j\xi_R)=Q_{e's,S_{2k,e'}}( w_{e'}^j\xi_R,w_{e'}^j\xi_R)
\le Q_{e',S_{2k,e'}}( w_{e'}^j\xi_R, w_{e'}^j\xi_R)\]
for $j=1,2$, and also
\[\begin{split}
&Q_{u,\Omega}(\tilde\xi_{e'}\xi_R,\tilde \xi_{e'}\xi_R)=kQ_{u,S_{2k,e'}}(\xi_{e'}\xi_R, \xi_{e'}\xi_R)\\
&=k
A_{e'}^2Q_{u,S_{2k,e'}}(w_{e'}^1\xi_R, w_{e'}^1\xi_R)+kB_{e'}^2Q_{u,S_{2k,e'}}(w_{e'}^2\xi_R, w_{e'}^2\xi_R)\end{split}\]
Combining this with \eqref{eq:limsup} we obtain
\[\limsup_{R\to \infty} Q_{u,\Omega}(\tilde \xi_{e'}\xi_R,\tilde \xi_{e'}\xi_R)\leq 0\]
which, together with \eqref{eq:cs-particular} gives 
\[\int_{\Omega}\nabla \tilde \xi_{e'}\nabla \phi-f'(|x|,u)\tilde \xi_{e'}\phi =\lim_{R\to \infty}Q_{u,\Omega}(\tilde \xi_{e'}\xi_R,\phi)=0 \]
Since $\phi\in C^{\infty}_0(\Omega)$ is arbitrary, then $\tilde \xi_{e'}$ solves $L_u\tilde \xi_{e'}=0$ and also $L_u\xi_{e'}=0$ in $S_{2k,e'}$. Then, defining the function $\widehat \xi_{e'}$ as in \eqref{hat-xi} we obtain 
as in the proof of Proposition \ref{prop-4.1} that $\xi_{e'}\equiv 0$, which is possibile, by \eqref{eq:xi-secondo}, since $w_{e'}^1$ and $w_{e'}^2$ have disjoint supports only when $A_{e'}=B_{e'}=0$, namely
\[0=Q_{u,S_{2k,e'}}(w_{e'}^j,\varphi_1) \ \ \text{ for }j=1,2.\]
The $k$-invariance of $\varphi_1$ and $\tilde w_{e'}$ then also implies that 
\[Q_{u,\Omega}(\tilde w_{e'}^j,\varphi_1) =k Q_{u,S_{2k,e'}}(w_{e'}^j,\varphi_1)=0\ \ \text{ for }j=1,2.\]
But this is not possible by the same argument in the end of the proof of Theorem 1.3, Case 1, in \cite{GPW}
and the contradiction concludes the proof of Case 1.\\
\noindent {\bf Case 2:} $S_*\neq \emptyset$\\
Let us assume that $e=e_\psi\in S_*$, so that $w_\psi\equiv 0$ in $S_{k,\psi}^+$. 
Since, by definition
\[ \inf _{\phi\in C^1_0(S_{k,e}^+)} Q_{u, S_{k,e}^+}(\phi,\phi)<0
\]
then, there exists $R_1>0$ such that 
\[\l_1(L_u,S_{k,\psi}^+\cap B_R)<0\
\ \text {for every }R>R_1.\]
For an $R>\max{R_0,R_1}$ we define, as in proof of Proposition \ref{prop-4.1}, the functions $g_\psi$, $ \widehat g_\psi$, $\tilde g_\psi$. By symmetry reasons $\tilde g_\psi$ is orthogonal to $\varphi_1$ in $L^2(\Omega\cap B_R)$ and satisfies
\[Q_{u,\Omega\cap B_R}(\tilde g_\psi,\tilde g_\psi)=2k\l_1(L_u,S_{k,\psi}^+\cap B_R)<0.\]
To conclude we want to prove that either $w_{\psi+\frac{\pi}{2k}}$ has a sign in $S_{k,\psi+\frac{\pi}{2k}}^+$ or $w_{\psi+\frac{\pi}{2k}}\equiv 0$ and
\begin{equation}\label{passo-verso-fine}
\inf _{\phi\in C^1_0(S_{k,\psi+\frac{\pi}{2k}}^+)} Q_{u, S_{k,\psi+\frac{\pi}{2k}}^+}(\phi,\phi)\geq 0
\end{equation}
since the thesis then follows by Propositions \ref{prop-5.1} and \ref{prop-5.2}. 
i) First we show that whenever $w_{\psi+\frac \pi{2k}}\equiv 0$ in $S_{k,\psi+\frac \pi{2k}}^+$ then \eqref{passo-verso-fine} holds. Suppose not, then for sufficiently large $R>\max{R_0,R_1}$ we have
\[\l_1(L_u,S_{k,\psi+\frac{\pi}{2k}}^+\cap B_R)<0\]
By the same arguments of the proof of Proposition \ref{prop-4.1} $i)$ the function $\tilde g_{\psi+\frac {\pi}{2k}}$ is orthogonal in $L^2(\Omega)$ to $\varphi_1$ and $\tilde g_\psi$, and satisfies
\[Q_{u,\Omega}(\tilde g_{\psi+\frac {\pi}{2k}},\tilde g_{\psi+\frac {\pi}{2k}})=2k\l_1(L_u,S_{k,\psi+\frac{\pi}{2k}}^+\cap B_R)<0\]
contradicting the bound $m_k(u)\leq 2$. This shows that, whenever $w_{\psi+\frac{\pi}{2k}}\equiv 0$ then \eqref{passo-verso-fine} should hold.\\
ii) It last to prove that, whenever $w_{\psi+\frac{\pi}{2k}}\neq 0$ then it has a sign in  $S_{k,\psi+\frac{\pi}{2k}}^+$.  Assume by contradiction it changes sign and define the functions $w_{e}^1$ and $w_{e}^2$ as in \eqref{eq:w1e2} relative to the direction $e=e_{\psi+\frac{\pi}{2k}}$. Then, as in the proof of Proposition \ref{prop-4.1} $ii)$ the function $\tilde \xi_e$ is orthogonal in $L^2(\Omega)$ to $\varphi_1$ and to $\tilde g_{\psi}$.
Then the proof follows as in the end of Case 1, where $\varphi_2$ is substituted by $\tilde g_\psi$ 
getting a contradiction and showing that, in this case, $_{\psi+\frac{\pi}{2k}}$ has a sign in $S_{k,\psi+\frac{\pi}{2k}}^+$.  
\end{proof}

\section{Applications}\label{se:6}
Let us consider the Lane-Emden problem
\begin{equation}\label{LE}
\left\{\begin{array}{ll}
-\Delta u = |u|^{p-1}u \qquad & \text{ in } \Omega, \\
u= 0 & \text{ on } \partial \Omega,
\end{array} \right.
\end{equation}
where $p>1$ is a real parameter. In the paper \cite{GGPS} the unique positive radial solution to \eqref{LE} is studied when $\Omega$ is an annulus. In particular problem \eqref{LE} admits a unique positive radial solution $u_p$ for every $p\in(1,\infty)$ and its Morse index converges to $+\infty$ as $p\to \infty$. This behavior produces infinitely many nonradial positive solutions that arise by bifurcation from the solution $u_p$. Moreover in \cite{G2} it has been proved that, for every
$k\in \N$, $k\geq 1$ there exists an exponent $p_k>1$ at which un unbounded continuum of $k$-invariant solutions to \eqref{LE} bifurcates. This continuum exists for every $p>p_k$ and all the solutions in it are symmetric with respect to a direction $e\in \mathcal S$ and are monotone in the angular variable in the sectors $S_{k,e}^\pm$, by construction. Namely they present the same monotonicity properties that we proved under the $k$-Morse index bound.\\
It is also proved in \cite{G2}, by Morse index considerations, that, letting 
\begin{equation}\label{H-k}
H^1_{0,k}:=\{v\in H^1_0(\Omega) :  v(x)=v(g(x)) \hbox{ for any }x\in \Omega,  \hbox{ for every } g\in \mathcal G_{\frac {2\pi}k}\}\end{equation}
the least energy solution $u_p^k$ in $H^1_{0,k}$, that can be found
 minimizing the Energy functional associated with \eqref{LE}
\[\mathcal E(u):=\frac 12 \int_\Omega|\nabla u|^2-\frac 1{p+1}\int_\Omega |u|^{p+1}\] 
on the $k$-invariant Nehari manifold
\[\mathcal N_k:=\{u\in H^1_{0,k}:  u\neq 0, \int_\Omega|\nabla u|^2=\int_\Omega|u|^{p+1}\}\]
is nonradial for every $p>p_k$, where $p_k$ is the same as before. \\
What it lasts to be proved is that these least energy solutions $u_p^k$ cannot have more invariances, namely they cannot belong to $H^1_{0,k'}$ for $k'> k$. \\
 But this can now follows from our results. Indeed a least energy solution $u_p^k\in H^1_{0,k} $ satisfies
\[m_k(u_p^k)=1
\]
so that, by Theorem \ref{teo:1}, $u_{p}^k$ is symmetric with respect to a direction $e\in \mathcal S$ in $S_{2k,e}$ and it is strictly monotone in the angular variable in $S_{k,e}^{\pm}$ so that cannot belong to $H^1_{0,k'}$ for $k'> k$.\\
Summarizing we have:
\begin{theorem}\label{teo-6.1}
For any $k\ge 1$ there exists an exponent $p_k$ for which $u_p^k$ is non radial
for any $p>p_k$, it is symmetric with respect to a direction $e\in \mathcal S$ and it is strictly monotone in the angular variable in $S_{k,e}^{\pm}$. Moreover for $p>p_k$ $u_p^k\neq  u_p^h$ with $h<k$ so that problem \eqref{LE} admits at least $k+1$ positive different solutions.
\end{theorem}
The $k+1$ solutions in Theorem \ref{teo-6.1} are given by the radial solution $u_p$, the least-energy solution $u_p^1$ (corresponding to $k=1$) which is the least-energy solution in $H^1_0(\Omega)$, the least energy solution $u_p^2\in H^1_{0,2}$ and so on until the least energy solution $u_p^k\in H^1_{0,k}$. \\

 This permits to say that the solutions found by bifurcation and the least energy solutions in the symmetric spaces $H^1_{0,k}$ possess the same symmetry and monotonicity properties in the sectors $S_{k,e}^\pm$ supporting the conjecture that indeed they are the same.\\
Finally we observe that solutions with a large number of peaks for large values of $p$ has been constructed in \cite{EMP} in a more general domain than an annulus via the Lyapunov Schmidt reduction method.  We believe that their solutions in the case of an annulus can coincide with ours $u_p^k$.

\

We are confident that a very similar result should hold also for the exponential nonlinearity, namely
\begin{equation}\label{EX}
\left\{\begin{array}{ll}
-\Delta u = \l e^u\qquad & \text{ in } \Omega, \\
u= 0 & \text{ on } \partial \Omega,
\end{array} \right.
\end{equation}
in an annulus, using the asymptotic behavior of the radial mountain pass solution $u_\l$ as $\l\to 0^+$, performed in \cite{GG}. 
From this it should follow that the Morse index of $u_\l$ converges to $+\infty$ as $\l\to 0^+$, generating infinitely many nonradial solutions that arise by bifurcation as in the previous case.
Applying the Mountain pass Theorem in the spaces $H^1_{0,k}$ then one ends with a $k$-Morse index one positive solution that cannot be radial by Morse index considerations. Then Theorem \ref{teo:1} applies and implies that solutions found in this way cannot coincide. 
Solutions with a large number of peaks for small values of $\l$ has been constructed in \cite{EGP} and \cite{DPKM} via the Lyapunov Schmidt reduction method.\\
Of course a result of this type deserves a deep study that we give to the interested reader.

\

\

Let us turn to the case of nodal solutions and consider now the Lane-Emden problem \eqref{LE} in the unit ball $B$. In \cite{GI} the radial nodal least-energy solution $u_p$ is studied. In particular it is shown that the Morse index of $u_p$ changes in its existence range $p\in (1,\infty)$, corresponding to some symmetric spaces $H^1_{0,k}$ when $k=3,4$ and $5$ giving rise to some bifurcating branches of solutions which arise from the nodal radial solution $u_p$. Along these branches solutions are symmetric with respect to a direction $e\in \mathcal S$ and are monotone in the angular variable in the sectors $S_{k,e}^\pm$, by construction. Namely they present the same monotonicity properties that we proved under the $k$-invariant Morse index bound.\\
It is also shown, using the $k$-Morse index, that, for $k=3,4$ and $k=5$, there exists an exponent $p_k$ such that least-energy nodal solutions in $H^1_{0,k}$, that we denote by $\widetilde u_p^k$, are nonradial when $p>p_k$, while it seems that they are radial when $p$ is near $p=1$. \\
Moreover it is known that $\tilde u_p^1$ is nonradial for every $p>1$, see \cite{BWW} or \cite{Pacella}, and it
has also been noticed in \cite{GI}, by Morse index considerations, that $\widetilde u_p^2$ is nonradial when $p$ is near $1$ and when $p$ is large.\\
These solutions $\widetilde u_p^k$ are found minimizing the energy functional $\mathcal E(u)$
on the nodal $k$-invariant Nehari manifold 
\[ \begin{array}{rl}\mathcal N_{k, nod}:=\Big\{u\in H^1_{0,k}:&   u^+\neq 0, \ \int_B|\nabla u^+|^2=\int_B |u^+|^{p+1}, \\ &  u^-\neq 0 , \ \int_B|\nabla u^-|^2=\int_B |u^-|^{p+1}\Big\} .\end{array} \]
($s^+$ ($s^-$) stands for  the positive (negative) part of $s$)
which has codimension $2$ by a result of \cite{BW}, so that their $k$-Morse index is exactly two, namely
\[m_k(\widetilde u_p^k)= 2.\] 
Then Theorem \ref{teo:2} applies when $p\ge2$, and hence a $k$-invariant nodal least energy solution $\widetilde u^k_p$ is symmetric with respect to a direction $e\in \mathcal S$ and strictly monotone in the angular variable in the sectors $S_{k,e}^{\pm}$, showing that $\widetilde u_p^1\neq \widetilde u_p^2\neq \widetilde u_p^3\neq \widetilde u_p^4\neq \widetilde u_p^5$ when they are nonradial and $p\ge2$. For $k=1$, since $H^1_{0,1}$ coincides with $H^1_0(B)$ then $\widetilde u_p^1$ coincide with the least energy nodal solution and hence it is foliated Schwarz symmetric for every $p$, by \cite{BWW}. Observe that the foliated Schwarz symmetry in the plane is nothing else that our $k$-invariance and $k$-monotonicity property for $k=1$.\\
Summarizing  previous and new results we have:
\begin{theorem}\label{teo:6.2}
The solution $\widetilde u_p^1$ is nonradial for every $p>1$. 
For $k=2,3,4,5$ there exists an exponent $p_k$ for which $\widetilde u_p^k$ is non radial
for any $p>p_k$. When they are nonradial, the functions $\widetilde u_p^k$ are symmetric with respect to a direction $e\in \mathcal S$ and strictly monotone in the angular variable in $S_{k,e}^{\pm}$. Moreover $\widetilde u_p^1\neq \widetilde u_p^k$ for every $p>1$ and every $k$, and for $p>p_k$ $\widetilde u_p^k\neq  \widetilde u_p^h$ with $h<k$ so that problem \eqref{LE} admits for $p>p_k$ at least $k+1$ distinct nodal solutions.
\end{theorem}
The $k+1$ in Theorem \ref{teo:6.2} solutions are given by the radial solution $u_p$, the nodal least-energy solution $\widetilde u_p^1$ (corresponding to $k=1$) which is the nodal least-energy solution in $H^1_0(\Omega)$, the nodal least-energy solution $\widetilde u_p^2\in H^1_{0,2}$ and so on until the nodal least-energy solution $\widetilde u_p^5\in H^1_{0,5}$. Finally, Morse index considerations in \cite{GI} suggest that $\widetilde u_p^k$ coincides with the radial nodal solution when $k\geq 6$.\\

As before we proved so far that  the solutions found by bifurcation and the least energy nodal solutions in the symmetric spaces $H^1_{0,k}$ possess the same symmetry and monotonicity properties in the sectors $S_{k,e}^\pm$ supporting the conjecture that indeed they are the same.\\
Solutions with this type of symmetry have been construct by Lyapunov Schmidt reduction method in \cite{EMP2}.

\

\

Theorem \ref{teo:2} can be applied also considering the $\sinh$-Poisson problem
\begin{equation}\label{SP}
\left\{\begin{array}{ll}
-\Delta u = \e (e^u-e^{-u})
 \qquad & \text{ in } \Omega, \\
u= 0 & \text{ on } \partial \Omega,
\end{array} \right.
\end{equation}
where $\e>0$ is a small parameter when $\Omega$ is a ball or an annulus 
and the solutions change sign. In this case the known results, see \cite{BPW}, suggest that nonradial solutions can be found using our minimization procedure in the spaces $H^1_{0,k}$. We leave the interested reader to carry out this study.

\

The previous results can be applied also to positive and nodal solutions of the H\'enon problem
\begin{equation}\label{H}
\left\{\begin{array}{ll}
-\Delta u = |x|^\a |u|^{p-1}u \qquad & \text{ in } \Omega, \\
u= 0 & \text{ on } \partial \Omega,
\end{array} \right.
\end{equation}
where $\Omega$ is a ball or an annulus, 
and can be used to distinguish solutions that belong to different spaces $H^1_{0,k}$. In the paper
\cite{AG}, as an example, it has been used, when $\Omega$ is a ball, in order to obtain some multiplicity results, minimizing the energy functional associated with \eqref{H}, namely
\[\mathcal E(u):=\frac 12 \int_B|\nabla u|^2-\frac 1{p+1}\int_B|x|^\a |u|^{p+1}\] 
on the $k$-invariant Nehari manifold
\[\mathcal N_k:=\{u\in H^1_{0,k}: \int_B|\nabla u|^2=\int_B |x|^\a|u|^{p+1}\}\]
or on the $k$-invariant  nodal Nehari manifold
\[ \begin{array}{rl}\mathcal N_{k, nod}:=\Big\{u\in H^1_{0,k}:&   u^+\neq 0, \ \int_B|\nabla u^+|^2=\int_B |x|^\a|u^+|^{p+1}, \\ &  u^-\neq 0 , \ \int_B|\nabla u^-|^2=\int_B |x|^\a|u^-|^{p+1}\Big\} .\end{array} \]
Minimizing $\mathcal E(u)$ on $\mathcal N_k$, for any $k\in \N$, $k\geq 1$, produces positive solutions that we denote by $u^k$ and satisfies
\[m_k(u^k)=1\]
 while minimizing $\mathcal E(u)$ on $\mathcal N_{k,nod}$,  for any $k\in \N$, $k\geq 1$, produces nodal solutions that we denote by $\widetilde u^k$ and satisfies
\[m_k(\widetilde u^k)= 2.\]
In \cite{AG} it has been shown, by a careful study of radial solutions, that:
\begin{theorem}[Theorem 4.2 in \cite{AG}]
Let $\a> 0$ be fixed. Then, there exists an exponent $p^\star_1(\a)>1$ such that problem \eqref{H} admits at least $j_1=1+\lceil \frac \a 2\rceil$ distinct positive solutions for every $p>p^\star_1(\a)$.
\end{theorem}
Here $\lceil t\rceil=\min\{k\in \Z : k\geq t \}$ denotes the ceiling function. 
One solution is radial while the others are not.
Further
\begin{theorem}[Theorem 4.4 in \cite{AG}]\label{teo:6.4}
Let $\a\geq 0$ be fixed. Then, there exists an exponent $p^\star_2(\a)>1$ such that problem \eqref{H} admits at least $j_2=1+\lceil \frac {2+\a} 2\kappa \rceil$ distinct nodal least energy solutions for every $p>p^\star_2(\a)$.
\end{theorem}
Here $\kappa \approx 5.1869$. This number $\kappa$ has been found out in \cite{GGP} in the study of nodal radial solutions to \eqref{LE} in the unit ball and indeed it is responsible of the existence of the nonradial nodal solutions $\widetilde u_p^k$ for $k=2,3,4$ and $5$. As seen in Theorem \ref{teo:6.4} it plays a role also in the case of nodal solutions to \eqref{H}, where its effects are combined with the ones of the parameter $\a$. When $\a=0$,  $\lceil \frac {2+\a} 2\kappa \rceil=5$ as said in Theorem \ref{teo:6.2}. \\
Positive solutions with this type of symmetry have been construct by the finite dimensional reduction method in \cite{EPW} for large values of $p$.\\
Nodal solutions with this type of symmetry have been construct by the finite dimensional reduction method in \cite{ZY} for large values of $p$.

\

We believe that very similar results should hold also for positive solutions with exponential nonlinearities of the H\'enon type 
\[-\Delta u=\l |x|^\a e^u \ \ \text{ in }\Omega\]
and for positive and sign changing solution of the $\sinh$-Poisson problem of H\'enon type
\[-\Delta u=\e |x|^\a(e^u-e^{-u}) \ \ \text{ in }\Omega\]
where $\l$ and $\e$ are small parameters. Note that Theorem \ref{teo:1} can be applied in the first case, while Theorem \ref{teo:2} holds in the second example.

We quote the existence results in \cite{GGN}, \cite{D1}, \cite{D2} 

\

We end observing that in the case of the unbounded domain $\R^2$, positive solutions to 
\[-\Delta u=|x|^\a e^u \ \ \text{ in }\R^2\]
have been classified in the famous paper \cite{PT} and they exhibit the same monotonicity properties that we have highlighted in our results.


\begin{thebibliography}{9999}


\bibitem[AP]{AP}{\sc A. Aftalion, F. Pacella}, Qualitative properties of nodal solutions of semilinear elliptic equations in radially symmetric domains, C. R. Acad. Sci. {\bf 339} (2004) 339-344
\bibitem[AG]{AG}{\sc A.L. Amadori, F. Gladiali}, On sign-changing solution to the H\'enon problem in the disc, preprint.
\bibitem[BPW]{BPW}{\sc T. Bartsch, A. Pisoia, T. Weth}, $N$-Vortex Equilibria for Ideal Fluids in Bounded Planar Domains and New Nodal Solutions of the sinh-Poisson and the Lane-Emden-Fowler Equations, Commun. Math. Phys. {\bf 297} (2010) 653-686 
\bibitem[BWe]{BW}{\sc T. Bartsch, T. Weth}, A note on additional properties of sign changing solutions to superlinear elliptic equations,
   Topological Methods in Nonlinear Analysis {\bf 22} (2003) 1-14
  \bibitem[BWW]{BWW} {\sc T. Bartsch, T. Weth, M. Willem,} Partial symmetry of least energy nodal solutions to
some variational problems,  J. Anal. Math. {\bf 96} (2005) 1-18
\bibitem[BN]{BN}{ \sc H. Brezis, L. Nirenberg}, Positive solutions of nonlinear elliptic equations involving critical Sobolev exponents. Comm. Pure Appl. Math., {\bf 36} (1983) 437-477
\bibitem[CL]{CL} {\sc W. Chen, C. Li}, Classification of solutions of some nonlinear elliptic equations, Duke Math. J. {\bf 63} (1991) 615-622
\bibitem[C]{C} {\sc C. V. Coffman} A nonlinear boundary value problem with many positive solutions. J. Differential Equations, {\bf 54} (1984) 429-437
\bibitem[DP]{DP} {\sc L. Damascelli, F. Pacella}, Symmetry results for cooperative elliptic systems via linearization, SIAM Journ. Math. Anal.
{\bf 45} (2013) 1003-1026
\bibitem[DGP1]{DGP1}{\sc L. Damascelli, F. Gladiali, F. Pacella}, A symmetry result for semilinear cooperative elliptic systems, Proc.Conf. Nonlinear Partial Differential Equations, AMS -Contemporary Mathematics Series {\bf 595} (2013)187-204
\bibitem[DGP2]{DGP2}{\sc L. Damascelli, F. Gladiali, F. Pacella}, Symmetry results for semilinear cooperative elliptic systems in unbounded domains, Indiana Univ. Math. Journal, {\bf 63} (2014) 615-649
\bibitem[D1]{D1}{\sc T. D'Aprile}, Multiple blow-up solutions for the Liouville equation with singular data,
Comm. Partial Differential Equations {\bf 38} (2013) 1409-1436
\bibitem[D2]{D2}{\sc T. D'Aprile}, Sign-changing blow-up solutions for H\'enon type elliptic equations with exponential nonlinearity,
J. Funct. Anal. {\bf 268} (2015) 2067-2101. 
\bibitem[DPKM]{DPKM} {\sc M. del Pino, M. Kowalczyk, M. Musso}, Singular limits in Liouville-type equations, Calc. Var. Partial Differential Equations {\bf 24} (2005) 47-81
\bibitem[EGP]{EGP} {\sc P. Esposito, M. Grossi, A. Pistoia}, On the existence of blowing-up solutions for a mean field equation, Ann. Inst.
H. Poincar\'e Analyse Non Lin\'eaire {\bf 22} (2005) 227-257
\bibitem[EMP]{EMP} {\sc P. Esposito, M. Musso, A. Pistoia}, 
Concentrating solutions for a planar elliptic problem involving nonlinearities with large exponent, Journal of Differential Equations, {\bf 227} (2006) 29-68
\bibitem[EMP2]{EMP2} {\sc P. Esposito, M. Musso, A. Pistoia}, On the existence and profile of nodal solutions for a two-dimensional elliptic problem with large exponent in nonlinearity, Proceedings of the London Mathematical Society {\bf 94} (2007) 497-519
\bibitem[EPW]{EPW} {\sc  P. Esposito, A. Pistoia, J. Wei, }
Concentrating solutions for the H\'enon equation in $R^2$
Journal d'Analyse Mathematique {\bf 100}, (2006), 249-280
\bibitem[GNN]{GNN}{\sc B. Gidas, W.M. Ni, L. Nirenberg}, Symmetry and related properties via the maximum principle, Comm. Math. Phys. {\bf 68} (3) (1979) 209-243 
\bibitem[GNN2]{GNN2}{\sc B. Gidas, W.M. Ni, L. Nirenberg}, Symmetry of positive solutions of nonlinear elliptic equations in $\R^N$, Mathematical Analysis and
Applications, Part A, in: Adv. in Math. Suppl. Stud., vol. 7a, Academic Press, New York, London, 1981, 369-402
  \bibitem[G]{G10}{\sc F. Gladiali}, {A global bifurcation result for a semilinear elliptic equation}, Journal of Math. Anal. and Appl.
{\bf 369} (2010)  306-311
\bibitem[G2]{G2}{\sc F.~Gladiali,} Separation of branches of $O(N-1)$-invariant solutions for a semilinear elliptic equation, Journal of Math. Anal. and Appl. {\bf 453} (2017) 159-173
\bibitem[GG]{GG}{\sc F. Gladiali, M. Grossi,} Singular limit of radial solutions in an annulus, Asymptot. Anal. {\bf 55 } (2007) 73-83
\bibitem[GGNe]{GGN} {\sc F. Gladiali, M. Grossi, S. Neves} Symmetry breaking and Morse index of solutions of nonlinear elliptic problems in the plane, Communications in Contemporary Mathematics {\bf 18} (2016)
\bibitem[GGPS]{GGPS}{\sc F. Gladiali, M. Grossi, F. Pacella, P.N. Srikanth}, Bifurcation and symmetry breaking for a class of semilinear elliptic equations in an annulus. Calc. Var. Partial Differential Equations {\bf 40} (2011) 295-317
\bibitem[GI]{GI}{\sc F. Gladiali, I. Ianni}, Quasiradial nodal solutions for the Lane-Emden problem in the ball, (2017), preprint. 
\bibitem[GPW]{GPW}{\sc F. Gladiali, F. Pacella, T. Weth}, Symmetry and nonexistence of low Morse index solutions in unbounded domains, J. Math. Pures Appl. {\bf  93} (2010) 536-558
\bibitem[GGP]{GGP}{\sc M.~Grossi, C.~Grumiau, F.~Pacella}, 
Lane Emden problems with large exponents and singular Liouville equations, 
  Journal des Mathematiques Pures et Appliquees, {\bf 101} (2014)  735-754
  \bibitem[L]{Li}{\sc Y. Y. Li},  Existence of many positive solutions of semilinear elliptic equations on annulus. J. Differential Equations, {\bf 83} (1990) 348-367
\bibitem[P]{Pacella}{\sc F. Pacella}, Symmetry results for solutions of semilinear elliptic equations with convex nonlinearities,  J. Funct. Anal. {\bf 192} (2002) 271-282
\bibitem[PW]{PW}{\sc F. Pacella, T. Weth}, Symmetry of solutions to semilinear elliptic equations via Morse index, Proc. Amer. Math. Soc. {\bf 135} (2007) 1753-1762
\bibitem[Pa]{Palais}{\sc R.S. Palais}, The Principle o f Symmetric Criticality, Commun. Math. Phys {\bf 69} (1979)19-30
\bibitem[PT]{PT}  {\sc J. Prajapat, G. Tarantello}, On a class of elliptic problems in $\R^2$: symmetry and uniqueness results,
{Proc. Royal Soc. Edinburgh: Sec. A}, {\bf 131} (2001) 967-985
\bibitem[SW]{SW} {\sc D. Smets, M. Willem}, Partial symmetry and asymptotic behavior for some elliptic variational problems, Calc. Var. Part. Diff. Eq. {\bf 18} (2003) 57-75
\bibitem[SSW]{SSW} {\sc D. Smets, M. Willem, J. Su},  Non-radial ground states for the H\'enon equation,  Commun. Contemp. Math. {\bf 4} (2002) 467-480
\bibitem[ZY]{ZY}{\sc     Y. B.  Zhang,  H. T. Yang}
Multi-peak nodal solutions for a two-dimensional elliptic problem with large exponent in weighted nonlinearity, 
Acta Mathematicae Applicatae Sinica, {\bf 31} (2015) 261-276 
\end{thebibliography}
\end{document}